\documentclass[12pt,reqno]{amsart}

\usepackage{amsfonts,amsmath,amsthm,amssymb}
\usepackage{cases}
\usepackage[usenames]{color}
\usepackage[dvipdf]{graphicx}

\usepackage{enumerate} 

\theoremstyle{plain}
\newtheorem{thm}{Theorem}[section]

\newtheorem{lem}[thm]{Lemma}

\newtheorem{prop}[thm]{Proposition}
\theoremstyle{remark}
\newtheorem{rem}{\bf{Remark}}
\numberwithin{equation}{section}

\newcommand{\N}{\mathbb{N}}

\newcommand{\R}{\mathbb{R}}
\newcommand{\bS}{\mathbb{S}}
\newcommand{\T}{\mathbb{T}}
\newcommand{\Z}{\mathbb{Z}}

\newcommand{\cF}{\mathcal{F}}

\newcommand{\AC}{{\rm AC\,}}

\newcommand{\BUC}{{\rm BUC\,}}

\newcommand{\Lip}{{\rm Lip\,}}

\newcommand{\al}{\alpha}
\newcommand{\gam}{\gamma}
\newcommand{\del}{\delta}
\newcommand{\ep}{\varepsilon}
\newcommand{\kap}{\kappa}
\newcommand{\lam}{\lambda}
\newcommand{\sig}{\sigma}
\newcommand{\om}{\omega}

\newcommand{\Gam}{\Gamma}
\newcommand{\Lam}{\Lambda}
\newcommand{\Om}{\Omega}

\newcommand{\ol}{\overline}
\newcommand{\ul}{\underline}
\newcommand{\pl}{\partial}
\newcommand{\supp}{{\rm supp}\,}

\newcommand{\dist}{{\rm dist}\,}
\newcommand{\Div}{{\rm div}\,}

\newcommand{\tr}{{\rm tr}\,}

\begin{document}
\title[Asymptotic speed of solutions to birth and spread type nonlinear PDEs]
{Existence of asymptotic speed of solutions to birth and spread type nonlinear partial differential equations}

\date{\today}

\thanks{
The work of YG was partially supported by Japan Society for the Promotion of Science (JSPS) through grants KAKENHI \#26220702,  \#16H03948.
The work of HM was partially supported by KAKENHI \#15K17574, \#26287024, \#16H03948.
The work of HT was partially supported by NSF grant DMS-1664424.}

\author[Y. Giga]{Yoshikazu Giga}
\address[Y. Giga]{
Graduate School of Mathematical Sciences, 
University of Tokyo 
3-8-1 Komaba, Meguro-ku, Tokyo, 153-8914, Japan}
\email{labgiga@ms.u-tokyo.ac.jp}

\author[H. Mitake]{Hiroyoshi Mitake}
\address[H. Mitake]{
Graduate School of Mathematical Sciences, 
University of Tokyo 
3-8-1 Komaba, Meguro-ku, Tokyo, 153-8914, Japan}
\email{mitake@ms.u-tokyo.ac.jp}

\author[T. Ohtsuka]{Takeshi Ohtsuka}
\address[T. Ohtsuka]{
Division of Mathematical Sciences, Graduate School of Engineering, 
Gunma University, 4-2 Aramaki-cho, Maebashi, 371-8510, Japan}
\email{tohtsuka@gunma-u.ac.jp}

\author[H. V. Tran]{Hung V. Tran}
\address[H. V. Tran]
{Department of Mathematics,
University of Wisconsin,
480 Lincoln Dr.,
Madison, WI 53706, USA.}
\email{hung@math.wisc.edu}



\begin{abstract}
In this paper, we prove the existence of asymptotic speed of solutions to  
fully nonlinear, possibly degenerate parabolic partial differential equations in a general setting. 
We then give some explicit examples of equations in this setting and study further properties of the asymptotic speed for each equation.
Some numerical results concerning the asymptotic speed are presented.
\end{abstract}

\maketitle


\section{Introduction}
This is a continuation of \cite{GMT}, where we discussed large time average of solutions to a model equation in the crystal growth theory to be described in Section \ref{sec:birth-spread}.
Motivated by this work, in this paper, we study a fully nonlinear, possibly degenerate parabolic partial differential equation (PDE) of the type
\begin{numcases}
{{\rm(C)}\qquad}
u_t+F(Du,D^2u)
=f(x) & in $\R^n\times(0,\infty)$, \nonumber \\
u(\cdot,0)=u_0 & on $\R^n$,  \nonumber 
\end{numcases}
where $u:\R^n\times[0,\infty)\to\R$ is a unknown function, and 
$u_t$, $Du$ and $D^2 u$ denote the time derivative, the spatial gradient and Hessian of $u$, respectively. 
Here $F:(\R^n\setminus\{0\}) \times \bS^n \to\R$ is a given continuous function, where $\bS^n$ denotes the space of $n \times n$ 
real symmetric matrices. 
We assume further that $F$ is \textit{degenerate elliptic}, that is,
\[
F(p,X+Y) \leq F(p,X) \quad \text{ for all } p \in \R^n \setminus \{0\}, \ X,Y \in \bS^n \text{ with } Y \geq 0,
\]
and $F_\ast(0,0)=F^\ast(0,0)=0$, where we denote by $F_\ast, F^\ast$ 
 the upper and lower semicontinuous envelope of $F$, respectively (see \cite{CIL, G-book} for definitions). 
Typical examples of $F$  we have in our mind are the ones appearing in the level set approach for surface evolution equations. 

The function $f:\R^n \to [0,\infty)$ on the right hand side of (C) is called a \textit{source term} in the paper, 
which is assumed to be Lipschitz continuous and 
have a compact support. The following condition is often set in the paper
\begin{equation}\label{f-con}
f \in C^1_c(\R^n) \text{ and there exists $R_0>0$ such that $\supp(f) \subset B(0,R_0)$.}
\end{equation}
We also suppose that the initial condition $u_0:\R^n\to\R$ is in $\BUC(\R^n)$, 
where $\BUC(\R^n)$ is the set of bounded uniformly continuous functions on $\R^n$.
We are always concerned with viscosity solutions in this paper, and the term ``viscosity" is omitted henceforth.

The well-posedness (existence, comparison principle and stability results) for (C) is well established
in the theory of viscosity solutions under suitable assumptions (see \cite{CIL, G-book} for instance). 
Our main goal in this paper is to study the large time average of $u$ as $t \to \infty$, that is,
\begin{equation}\label{goal}
\lim_{t\to \infty} \frac{u(x,t)}{t}, \qquad \text{ for each 
} x \in \R^n.
\end{equation}
We call the limit in \eqref{goal} the \textit{asymptotic speed} of the solution $u$ to (C) if it exists. 
This question is important as a starting point to study qualitative and quantitative behaviors of $u(x,t)$ as $t \to \infty$.

A common strategy in the literature to obtain \eqref{goal} is to construct appropriate barriers 
by using subsolutions and supersolutions to (C) which have the same asymptotic speed. 
Let us briefly 
describe this strategy in periodic homogenization theory
by assuming that $f$ is $\Z^n$-periodic instead of \eqref{f-con} for the moment.
Note that under this periodic situation, 
\eqref{f-con} does not hold unless $f \equiv 0$.
Because of the periodic structure, one is able to study the following ergodic (cell) problem
\[
{\rm (E)} \qquad F(Dv,D^2v) = f(x) +c \quad \text{ in } \T^n:= \R^n/\Z^n.
\]
Here $(v,c) \in C(\T^n) \times \R$ is a pair of unknowns.
Under some appropriate assumptions, we can show that there exists a unique constant $c\in \R$
so that (E) has a solution $v\in C(\T^n)$ (see \cite{LS} for example). 
This yields that
$v(x) + C -ct$ is a solution to (C) with initial data $v(x)+C$ for any $C\in \R$.
Set $C_1 =\|v\|_{L^\infty(\T^n)} +\|u_0\|_{L^\infty(\R^n)}$.
By the comparison principle for (C), it is straightforward to see that
\[
v(x) - C_1 - ct \leq u(x,t) \leq v(x)+C_1 -ct \quad \text{ for all } (x,t) \in \R^n \times [0,\infty),
\]
which clearly 
implies that
\[
\lim_{t \to \infty}  \frac{u(x,t)}{t} = -c \quad \text{ for each 
} x \in \R^n.
\]
The simple fact that 
$\| v \|_{L^\infty (\mathbb{R}^n)}=\| v \|_{L^\infty (\mathbb{T}^n)} < \infty$
plays a crucial role here.

As already noted, the source term $f$ in this paper satisfies \eqref{f-con} and is compactly supported in $B(0,R_0)$, 
which means that (C) does not have a periodic structure of any sort
and that there is no corresponding cell/ergodic problem in a compact set.
The above approach (though quite natural and general) therefore is not applicable in this setting.

In this paper, we develop a method to prove the existence of the asymptotic speed for solution of (C) under quite general assumptions ((A1)--(A3) in Section \ref{sec:existence}) in Theorem \ref{thm:main}. 
We put this in an abstract framework as our approach is quite general. 
A key point is to keep track of $m(t)=\sup_{x\in \R^n} u(x,t)$ for $t\geq 0$
and show that $m$ is subadditive in time $t$. 
A similar idea in the periodic setting appears in \cite[Section 10.3]{B}. 
Then, in Section \ref{sec:application}, we show that (A1)--(A3) hold true for three classes of equations by deriving a global Lipschitz bounds: first-order Hamilton-Jacobi equations, a forced mean curvature flow (a crystal growth model), and a truncated inverse mean curvature flow (a volcano formation model).
Thus, we have existence of asymptotic speed for solutions to these equations. 
For overview of the theory of large time behavior for fully nonlinear equations 
the readers are referred to \cite{B, LMT, Giga-ICM}, and papers cited there.

In Section \ref{sec:speed}, we study qualitative properties of asymptotic speed of each equation pointed out in Section \ref{sec:application}. 
The asymptotic speed of solutions to first-order Hamilton-Jacobi equations and truncated inverse mean curvature flow is completely characterized in Subsection \ref{subsec:V-pos}.
For forced mean curvature flow, it is harder to analyze the asymptotic speed of its solution.
In the radially symmetric setting, we give a complete and satisfactory characterization  in Subsection \ref{subsec:rad}.  
In non-radially symmetric settings, 
the situation seems much more complicated and 
we obtain some partial results in Subsection \ref{subsec:non-rad}. 
We present some numerical results in Section \ref{sec:numerical}, 
and indicate further questions 
which still remain open.
Finally, in Appendix, we introduce a volcano formation model, and discuss some background on inverse mean curvature flow.
Some typical results of this paper have been announced in \cite{Giga-ICM}.

\subsection*{Acknowledgement} The authors thank Professors Fumio Nakajima for giving us papers \cite{M1878, B1885} by J.\ Milne and G.\ F.\ Becker as well as useful information on volcanoes' shapes.
 The authors also thank Professor Takehiro Koyaguchi for giving us valuable information on modern theory of volcanoes' formation.

\section{Birth and spread type nonlinear PDEs}\label{sec:birth-spread}
In this section, we derive a PDE of the form of (C) as a continuum limit of a birth and spread type model in the theory of crystal growth (see \cite[Section 2.6]{OR} for instance). 
From the continuum point of view, it is derived from the Trotter-Kato approximation as following.
Let $u_0$ be the given height of a crystal surface
at initial time immersed in a supersaturated medium.
We assume that there is no dislocation in the crystal lattice. 
Then, the crystal grows according to the following processes.  
\begin{enumerate}
 \item Birth: adatoms 
       (atoms on the surface) make a small ``terrace''
       on the surface by their concentration.
 \item Spread: the terraces evolve by catching adatoms.
\end{enumerate}
Fix a small time step $\tau>0$.
Within time $\tau$, the birth process starts in $B(0,R_0)$ 
with supersaturation rate $f(x)$ at each $x\in B(0,R_0)$,
and the crystal surface evolves vertically as 
the graph of $v_0 = u_0(\cdot) + \tau f(\cdot)$.
During the next short time $\tau$, 
the terraces evolve 
horizontally by the spread process,
which is described by the evolution of 
each level set of $v_0$
by the surface evolution equation 
\begin{equation}\label{eq:surface}
V=g(n(x),\kap(x)). 
\end{equation}
More precisely, let 
$D_0 = \{ x \in \mathbb{R}^n \,:\, \ v_0 (x) > c \}$
be a terrace of the surface at height $c \in \mathbb{R}$ 
and $\Gam_0 = \partial D_0 $
be 
the edge of the terrace after the birth process.
Then, the terrace evolves horizontally as 
$D_0^\tau = \{ x \in \mathbb{R}^n \,:\, \mathrm{dist} (x, D_0) < \tau V \}$
with \eqref{eq:surface}, where $\dist(x, D_0):=\inf\{|x-y| \,:\, y\in\pl D_0\}$ if $x\in \R^n\setminus \ol{D}_0$ and 
$\dist(x, D_0):=-\inf\{|x-y| \,:\, y\in\pl D_0\}$ if $x\in D_0$. 
Here, 
$g$ is a given function and $V$ is the outward normal velocity of $\Gam_0$. 
The functions $n(x)$, $\kappa(x)$, respectively, are the outer unit normal and mean curvature of $\Gam_0$ at $x \in \Gam_0$.
In the spread process, 
the above evolution occurs for all height $c \in \mathbb{R}$
as in the continuum sense,
and we obtain the height $u(\cdot, 2 \tau)$
of the crystal surface at $t = 2 \tau$.
Then, we return to the birth process by resetting
$u_0 = u(\cdot, 2 \tau)$.
By repeating 
the above processes, 
we obtain the Trotter-Kato approximation.

Let us describe this in a clear mathematical framework with a double-step method.
Consider two initial value problems 
\begin{numcases}
{{\rm (N)}\quad}
v_t=f(x)
 & in $\R^n\times(0,\infty)$, \nonumber \\
v(\cdot,0)=u_0 & in $\R^n$, \nonumber 
\end{numcases}
and 
\[
{\rm(P)}\quad 
\left\{
\begin{array}{ll}
\displaystyle
w_t=g\left(-\frac{Dw}{|Dw|}, \Div\Big(\frac{Dw}{|Dw|}\Big)\right)|Dw| 
 & \text{in} \ \R^n\times(0,\infty), \nonumber \\
w(\cdot,0)=u_0 & \text{in} \ \R^n, \nonumber 
\end{array}
\right.
\]
where $g:\R^n\times\R\to\R$ is a given function. 
Notice that the equation in (P) is the level set equation (see \cite{G-book}) for the surface evolution equation  $V=g(n(x),\kap(x))$ on $\Gam_t$. 
Under some suitable assumptions on $g$, the well-posedness for (P) holds.

We call (N) and (P) the \textit{nucleation problem} and the \textit{propagation problem}, respectively. 
Define the operators $S_1(t), S_2(t):\Lip(\R^n)\to\Lip(\R^n)$ by 
\begin{equation}\label{TK-1}
S_1(t)[u_0]:=u_0(\cdot)+t f(\cdot), \quad\text{and}\quad
S_2(t)[u_0]:=w(\cdot,t),   
\end{equation}
where $w$ is the unique viscosity solution of (P) with given initial data $u_0$. 
For $x\in\R^n$, small time step  $\tau>0$, and $i\in\N$, set 
\begin{equation}\label{TK-formula}
U^\tau(x,i\tau):=S_1(\tau)\big(S_2(\tau)S_1(\tau)\big)^{i}[u_0](x).  
\end{equation}
This is called the \textit{Trotter-Kato product formula} with value function $U^\tau(x, i\tau)$.
By using a general framework in \cite{BS}, for $t=i\tau>0$ fixed, uniqueness and stability yield that
\begin{equation}\label{TK:limit}
\lim_{\substack{i \to \infty\\ i\tau=t}} U^\tau(x,i\tau)= u(x,t) \quad\text{locally uniformly for} \ x\in\R^n, 
\end{equation}
and $u$ is the unique viscosity solution to (C) with 
\[
F(p,X):=- g\left(-\frac{p}{|p|}, \frac{1}{|p|}\tr\left(\left(I_n - \frac{p\otimes p}{|p|^2}\right)X \right)\right)  |p|, 
\] 
where $I_n$ is the identity matrix of size $n$, and for $Y \in \mathbb S^n$, $\tr Y$ denotes the trace of $Y$.
We call this equation a \textit{birth and spread type nonlinear partial differential equations} in the paper.  
From the derivation, we can see that the equation has double nonlinear effects coming from the interaction of the nucleation and the surface evolution. 

It is worth emphasizing that the geometric structure of the operator for $F$ is not necessarily required to obtain convergence result \eqref{TK:limit}. 
We have restricted the propagation problem to equation (P) just to simplify explanation of a birth and spread model.


\section{Existence of asymptotic speed} \label{sec:existence}
In this section, we provide a simple way to prove the existence of the asymptotic speed for the solution $u$ of (C) in an abstract way. 
The assumptions we put are in the general abstract form, which will be verified later for each situation.
In particular, we do not need to assume \eqref{f-con} here.
We assume the followings.
\begin{itemize}
\item[(A1)]  The comparison principle holds for (C) in the class of bounded functions on $\R^n \times [0,T]$ for each $T>0$. 
Moreover, for any given initial data $u_0 \in \BUC(\R^n)$, (C) has a viscosity solution $u \in C(\R^n \times [0,\infty))$ which is bounded on $\R^n \times [0,T]$ for each $T>0$.

\item[(A2)] For $u_0 \equiv 0$, the solution $u$ to (C) is uniformly continuous in the space variable $x$ for all $t\geq 0$, that is, 
there exists a continuous, increasing function $\omega:[0,\infty)\to[0,\infty)$ with $\omega(0)=0$ such that 
\[
|u(x,t)-u(y,t)|\le\omega(|x-y|)\quad\text{for all} \ x,y\in\R^n, t\ge0.
\]


\item[(A3)] For $u_0 \equiv 0$, let $u$ be the corresponding solution to (C). 
There exists $R_0>0$ such that for each $T>0$, we have 
\[
u(x_T,s_T) = \max_{\R^n \times [0,T]} u \quad\text{for some} \ (x_T,s_T) \in \ol{B}(0,R_0) \times [0,T].
\]
\end{itemize}

Let us first give a few comments about assumptions (A1)--(A3). 
While (A1)--(A2) are quite standard in the theory of viscosity solutions, (A3) looks a bit restrictive.
This turns out to be natural if we assume that $f$ satisfies \eqref{f-con} thanks to the maximum principle
and  the fact that $F$ is independent of $x$.

\begin{lem}\label{lem:f-A3}
Assume that {\rm (A1)} and \eqref{f-con} hold. Then {\rm (A3)} is valid.
\end{lem}

\begin{proof}
If $f\equiv 0$, then $u \equiv 0$ and there is nothing to prove.
We hence may assume that $f \not \equiv 0$. It is clear then that $u \geq 0$ and $u \not \equiv 0$.

Fix $T>0$ and set $\sig = \sup_{\R^n \times [0,T]} u >0$.
For $\ep,\del>0$ sufficiently small, there exists $(x_{\ep,\del}, t_{\ep,\del}) \in \R^n \times (0,T]$ such that
\[
u(x_{\ep,\del}, t_{\ep,\del}) = \max_{\R^n \times [0,T]} \left( u(x,t) - \ep t - \del (|x|^2+1)^{1/2} \right)>0.
\]
By the definition of viscosity subsolution, we have
\[
\ep + F\left(\del \frac{x_{\ep,\del}}{(|x_{\ep,\del}|^2+1)^{1/2}}, \del \frac{(|x_{\ep,\del}|^2+1)I_n- x_{\ep,\del} \otimes x_{\ep,\del}}{(|x_{\ep,\del}|^2+1)^{3/2}} \right) \leq f(x_{\ep,\del}).
\]
Let $\del \to 0$ first to deduce that $(x_{\ep,\del},t_{\ep,\del}) \to (x_\ep,t_\ep)$ by passing to a subsequence if necessary and $x_\ep \in \ol{B}(0,R_0)$
as $f=0$ on $\R^n \setminus B(0,R_0)$.
We then let $\ep \to 0$ to get the desired result.
\end{proof}

\begin{lem}\label{lem:A1-A2}
Assume that {\rm (A1)} holds. Let $u$ be the solution to {\rm (C)} with the initial data $u_0 \equiv 0$.
Then, $u$ is Lipschitz in time, and
\[
\|u_t\|_{L^\infty(\R^n \times [0,\infty))} \leq M,
\]
where $M=\max_{\R^n} f$.
\end{lem}

\begin{proof}
 It is clear that $\varphi(x,t)=Mt$ for $(x,t)\in \R^n \times [0,\infty)$ is a supersolution to (C) 
 because of the fact that $F_\ast(0,0)=F^\ast(0,0)=0$.
 We use the comparison principle to get
\begin{equation}\label{bound-u}
0 \leq u(x,t) \leq Mt \quad \text{ for all } (x,t) \in \R^n \times [0,\infty).
\end{equation}
Thus, $\|u_t(\cdot,0)\|_{L^\infty(\R^n)} \leq M$.

For any given $s>0$, both $(x,t) \mapsto u(x,t+s)$ and $(x,t) \mapsto u(x,t)$ are viscosity solutions to (C)
with initial data $u(\cdot,s)$ and $u(\cdot,0)$, respectively. 
By the comparison principle in (A1) and \eqref{bound-u}, 
\[
\|u(\cdot,t+s) - u(\cdot,t)\|_{L^\infty(\R^n)} \leq \|u(\cdot,s) - u(\cdot,0)\|_{L^\infty(\R^n)} \leq Ms.
\]
Divide both sides of the above by $s$ and let $s\to 0+$ to get the conclusion.
\end{proof}

Below is one of our main results of this paper on the existence of asymptotic speed.
\begin{thm} \label{thm:main}
Assume that {\rm (A1)--(A3)} hold. 
Assume also that $M=\max_{\R^n}f$ exists and finite.
Let $u$ be the solution to {\rm (C)} with a given initial data $u_0 \in \BUC(\R^n)$.
There exists $c \in[0,M]$ such that
\begin{equation}\label{u-conv}
\lim_{t \to \infty} \frac{u(x,t)}{t} = c \quad \text{ locally uniformly for } x \in \R^n.
\end{equation}
Furthermore, $c$ is independent of the choice of $u_0$.
\end{thm}

\begin{proof}
Since the comparison principle holds, in order to prove \eqref{u-conv},
we can assume that $u_0 \equiv 0$. 
Recall that \eqref{bound-u} gives us
\[
0=u_0(x) \leq u(x,t) \leq Mt.
\]

For $t \ge 0$, set $m(t) = \sup_{x\in \R^n} u(x,t)$. We now show that
\begin{equation}\label{sub-add}
m(t+s) \leq m(t) +m(s) \quad \text{ for all } s,t \geq 0.
\end{equation}
Fix $s \geq 0$. We note that $(x,t) \mapsto v(x,t)=u(x,t+s)-m(s)$ and $(x,t) \mapsto u(x,t)$ are both solutions to (C),
and
\[
v(x,0) = u(x,s) - m(s) \leq 0 = u(x,0).
\]
Thus, $v(x,t) \leq u(x,t)$ in light of the comparison principle. In particular, we get that \eqref{sub-add} holds,
which means that $m$ is subadditive on $[0,\infty)$.
By Fekete's lemma, there exists $c \in [0,\infty)$ such that
\begin{equation}\label{lim-m}
\lim_{t \to \infty} \frac{m(t)}{t} = c=\inf_{s>0} \frac{m(s)}{s}.
\end{equation}
It is clear that $c \leq M$ because of \eqref{bound-u}.
If $c=0$, then \eqref{u-conv} holds immediately. 
We therefore only need to consider the case that $c>0$.
Fix $\ep>0$. There exists $T=T(\ep)>0$ such that
\[
c \leq \frac{m(t)}{t} \leq c+\ep \quad \text{ for all } t>T.
\]
For $t>\max\{\frac{MT}{c},M\}$, we use (A3) to have that
\begin{equation}\label{lower-s-t}
ct \leq \max_{\R^n \times [0,t]} u = u(x_t,s_t) \leq M s_t 
\quad\text{for some} \ (x_t,s_t) \in \ol{B}(0,R_0) \times [0,t],
\end{equation}
which implies that $s_t \geq \frac{ct}{M} \ge T$. Thus, we are able to improve \eqref{lower-s-t} as
\begin{equation}\label{better-s-t}
ct \leq \max_{\R^n \times [0,t]} u = u(x_t,s_t) \leq (c+\ep) s_t,
\end{equation}
which yields $s_t \geq \frac{c}{c+\ep}t$. So for any $x\in B(0,R)$ for $R>0$ given, we use (A2) and Lemma \ref{lem:A1-A2} to estimate that
\[
|u(x,t) - u(x_t,s_t)| \leq\omega(|x-x_t|) + C|t-s_t| \leq \omega(R+R_0) + \frac{C\ep t}{c+\ep}.
\]
Hence, for $t>\max\{\frac{MT}{c},M\}$,
\[
c - \frac{C\ep}{c+\ep} - \frac{\omega(R+R_0)}{t} \leq \frac{u(x,t)}{t} \leq c + \ep.
\]
The proof is complete.
\end{proof}

\begin{rem}
We note that the use of the Fekete lemma is quite natural in the literature once some subadditive quantities are identified.
A similar argument in the periodic setting appeared in a lecture note of Barles \cite{B} (see Section 10.3, the proof of Theorem 10.2). 
This is exactly the setting described in Introduction, and we can also obtain the same result by using the cell/ergodic problem.
In general, the lack of periodicity prevents us from using the natural compactness property of $\T^n$.
In a sense, (A3) is a compactness assumption, which  is a simple and effective replacement for the periodicity.
Furthermore, (A3) holds if we are in the periodic setting.
\end{rem}

We show next that (A1) also yields that 
if $u_0 \equiv 0$, then $u$ is Lipschitz continuous in the space variable $x$ on $\R^n \times [0,T]$
for each $T>0$, however, the Lipschitz constant $C$ depends on $T$ in this result.

\begin{prop}\label{prop:Lip-T}
Assume that {\rm (A1)} holds. 
Let $u$ be the solution to {\rm (C)} with a given initial data $u_0 \equiv 0$.
Then, for each $t>0$,
\[
|u(x_1,t) - u(x_2,t)| \leq \left(\|Df\|_{L^\infty(\R^n)} t \right) |x_1-x_2| \quad \text{ for all } x_1, x_2 \in \R^n.
\]
\end{prop}

\begin{proof}
Fix $y \in \R^n$. Let 
\[
\begin{cases}
v_-(x,t) = u(x+y,t) - \|Df\|_{L^\infty(\R^n)} |y|t \quad &\text{ for } (x,t) \in \R^n \times [0,\infty),\\
v_+(x,t) = u(x+y,t)  + \|Df\|_{L^\infty(\R^n)} |y|t \quad &\text{ for } (x,t) \in \R^n \times [0,\infty).
\end{cases}
\]
It is straightforward to see that $v_-$ and $v_+$ are a subsolution and a supersolution to (C) respectively, and
\[
v_-(x,0)= u_0(x) = v_+(x,0)=0.
\]
Therefore, the comparison principle in (A1) yields  $v_-(x,t) \leq u(x,t) \leq v_+(x,t)$. 
We thus have
\[
|u(x+y,t)-u(x,t)| \leq \|Df\|_{L^\infty(\R^n)} |y|t 
\quad\text{for all} \ x\in\R^n, t\ge0.
\qedhere
\]
\end{proof}

\begin{rem}
It is worth emphasizing that the Lipschitz bound obtained in Proposition \ref{prop:Lip-T} is not enough to obtain the existence of 
the asymptotic speed, and 
assumption (A2) plays an essential role in the proof of Theorem \ref{thm:main} (see the last part of the proof of Theorem \ref{thm:main}).
In fact, (A2) can be replaced by the following weaker assumption.

\begin{itemize}
\item[(A2)'] For $u_0 \equiv 0$, the solution $u$ to (C) is uniformly continuous in the space variable $x$ for each $t\geq 0$, that is, 
there exists a continuous, increasing function $\omega_t:[0,\infty)\to[0,\infty)$ with $\omega_t(0)=0$ such that 
\[
|u(x,t)-u(y,t)|\le\omega_t(|x-y|)\quad\text{for all} \ x,y\in\R^n.
\]
And for each fixed $R>0$,
\[
\lim_{t \to \infty} \frac{\om_t(R)}{t}=0.
\]
\end{itemize}
On the other hand, (A2) is easier to be verified than (A2)'.
We need to check (A2) carefully for each application in the next section.
\end{rem}


\section{Applications} \label{sec:application}
\subsection{First-order Hamilton-Jacobi equations}\label{subsec:HJ}
Assume that \eqref{f-con} holds and $F(p,X)= -H(p)$ where $H:\R^n \to \R$ is a continuous function  satisfying
\begin{equation}\label{HJ-con}
H(0)=0 \quad \text{and} \quad \lim_{|p| \to \infty} H(p)=+\infty.
\end{equation}
It is clear that if \eqref{HJ-con} holds, then we have the validity of (A1)--(A3) and hence also  of Theorem \ref{thm:main}.

A  special case is when $H$ is $1$-homogeneous, that is, $H(p) =  g\left(\frac{p}{|p|}\right)|p|$ for all $p \neq 0$ and $H(0)=0$.
Here, $g: \R^n \to (0,\infty)$ is a given continuous function.
This situation appears if we consider the surface evolution equation $V=g(n(x))$ in the birth and spread type model in 
Section \ref{sec:birth-spread}. 
Notice that $H$ is not necessarily convex. 

\subsection{Forced mean curvature flow}\label{subsec:FMC}
Consider a forced mean curvature flow 
\[
V=\kap+1
\]
in the birth and spread type model in Section \ref{sec:birth-spread}. 
Then, the associated PDE in (C) becomes
\begin{equation}\label{FMCF}
\begin{cases}
u_t -\left(\Div\left(\frac{Du}{|Du|}\right)+1\right)|Du|=f(x) \quad &\text{in} \ \R^n\times(0,\infty),\\
u(x,0)=u_0(x) \quad &\text{on } \R^n.
\end{cases}
\end{equation}
Assume  that \eqref{f-con} holds.
We have that (A1) holds (see \cite{G-book} for instance).
Therefore, we only need to verify (A2) here.

\begin{lem}\label{lem:lip}
Assume  that \eqref{f-con} holds. 
Let $u$ be the solution to \eqref{FMCF} with given initial data $u_0 \equiv 0$.
Then, $u$ is Lipschitz in space, and there exists $C>0$ depending only on $f$ and $n$ such that 
\[
\|Du\|_{L^\infty(\R^n\times[0,\infty))} \leq C. 
\]
\end{lem}

\begin{proof}
For $\ep\in (0,1)$, we consider the following approximated equation  
\begin{equation}\label{eq:approx}
\begin{cases}
u^{\ep}_t-\left(\Div\left(\frac{Du^{\ep}}{\sqrt{|Du^{\ep}|^2+\ep^2}}\right)+1\right)\sqrt{|Du^{\ep}|^2+\ep^2}-f=0 \quad &\text{in} \ \R^n\times(0,\infty),\\
u^\ep(x,0)=0 \quad &\text{on } \R^n.
\end{cases}
\end{equation}
This has a unique solution $u^\ep\in C^2_c(\R^n\times[0,\infty))$. 
Setting $b^{\ep}(p):=I_n-p\otimes p/(|p|^2+\ep^2)$,  
we rewrite \eqref{eq:approx} as 
\begin{equation} \label{approx-new}
u^\ep_t-b_{ij}^{\ep}(Du^\ep)u^\ep_{x_ix_j}-\sqrt{|Du^{\ep}|^2+\ep^2}-f=0 \quad \text{in} \ \R^n\times(0,\infty).
\end{equation}
Here we use Einstein's convention. 
 
We use the Bernstein method to get the gradient bound for $u^\ep$, hence $u$. 
Let $w^\ep:=|Du^\ep|^2/2$.
Differentiate the above equation with respect to $x_k$ and multiply by 
$u^\ep_{x_k}$ to yield 
\[
w^\ep_t -b_{ij}^{\ep}\left(w^{\ep}_{x_ix_j}-u^\ep_{x_jx_k}u^\ep_{x_ix_k}\right)-Df\cdot Du^\ep
-u^\ep_{x_i x_j} D_p b_{ij}^{\ep}\cdot Dw^\ep+\frac{Du^\ep\cdot Dw^\ep}{\sqrt{|Du^{\ep}|^2+\ep^2}}=0.
\]
Fix $T>0$.
Take $(x_0,t_0)\in \R^n\times (0,T]$ so that $w^\ep(x_0,t_0)=\max_{\R^n\times[0,T]}w^\ep$. 
At this point, we have 
\begin{equation}\label{Bern-1}
b_{ij}^{\ep}u^\ep_{x_jx_k}u^\ep_{x_ix_k}-Df\cdot Du^\ep
\le0. 
\end{equation}
By using a modified Cauchy-Schwarz inequality  (see Remark \ref{rem:CS} below)
\begin{equation}\label{CS-ineq}
 (\tr AB)^2\le \tr(ABB)\tr A\quad \text{for all} \ A, B\in \bS^n,  \  A\ge0,
\end{equation}
we obtain 
\begin{equation}\label{Bern-2}
Df\cdot Du^\ep\ge 
\tr(b^\ep(Du^\ep)D^2u^\ep D^2u^\ep)\ge \frac{\left(\tr(b^\ep(Du^\ep)D^2u^\ep)\right)^2}{\tr(b^\ep(Du^\ep))}
\ge \frac{\left(\tr(b^\ep(Du^\ep)D^2u^\ep)\right)^2}{n}.  
\end{equation}

By repeating the proof of Lemma \ref{lem:A1-A2}, 
we have that $\|u^\ep_t\|_{L^\infty(\R^n \times [0,\infty))}\le M+1$, where $M=\max_{\R^n} f$, for all $\ep \in (0,1)$. 
We use this and \eqref{approx-new} to yield
\begin{equation}\label{Bern-3}
\left(\tr(b^\ep(Du^\ep)D^2u^\ep)\right)^2=
\left(u^\ep_t-\sqrt{|Du^{\ep}|^2+\ep^2}-f\right)^2
\ge \frac{1}{2}|Du^\ep|^2-C,
\end{equation}
where $C=4(2M+1)^2$. 
 
Combining \eqref{Bern-2} and \eqref{Bern-3} together, we obtain 
\[
\frac{1}{2}|Du^\ep|^2-C \leq n Df\cdot Du^\ep \le C |Du^\ep|, 
\] 
which implies that $\|Du^\ep\|_{L^\infty(\R^n \times [0,\infty))} \leq C$ for some $C>0$ depending only on $\|f\|_{L^\infty}$, $\|Df\|_{L^\infty}$, and $n$.
Let $\ep \to 0$ to yield the desired result.
\end{proof}

\begin{rem}\label{rem:CS}
We give a simple proof of \eqref{CS-ineq} here. 
By the Cauchy-Schwarz inequality, we always have 
\[
0\le \left(\tr(ab)\right)^2\le \tr(a^2)\tr(b^2)\quad\text{for all} \ a,b\in \bS^n. 
\] 
For $A, B\in \bS^n$ with $A\ge 0$, set $a:=A^{1/2}$ and $b:=A^{1/2}B$. Then, 
\[
(\tr(AB))^2\le \tr(A)\tr(A^{1/2}BA^{1/2}B)=\tr(A)\tr(ABB). 
\]
\end{rem}

\subsection{Truncated inverse mean curvature flow}\label{subsec:volcano}
Consider a truncated normal velocity
\begin{equation}\label{surface:t-inv}
V=\frac{1}{\chi(\kap)}   
\end{equation}
in the birth and spread type model in Section \ref{sec:birth-spread},  
where we set 
\begin{equation}\label{func:truncate}
\chi(r):=\min\{\max\{r, \lam\}, \Lam\} \quad \text{ for } r \in \R. 
\end{equation}
Here $\lam$ is sufficiently small and $\Lam$ is sufficiently large satisfying $0<\lam<\Lam$ are given constants.
Then, the associated PDE in (C) becomes 
\begin{equation}\label{eq:truncate}
\begin{cases} \displaystyle
u_t-\frac{|Du|}{\chi\left(\frac{-\tr\left(b(Du)D^2u\right)}{|Du|}\right)} =f(x)  \quad &\text{in } \R^n \times (0,\infty),\\
u(x,0)=u_0(x) \quad &\text{on } \R^n,
\end{cases}
\end{equation}
where $b(p)=I_n - p\otimes p/|p|^2$.

\begin{lem}
Assume that \eqref{f-con} holds.
Let $u$ be the solution to \eqref{eq:truncate} with given initial data $u_0 \equiv 0$. 
Then, $u$ is Lipschitz in space, and there exists $C>0$ depending only on $f$ and $\Lam$ such that 
\[
\|Du\|_{L^\infty(\R^n\times[0,\infty))} \leq C.  
\] 
\end{lem}
\begin{proof}
Note first that $\|u_t\|_{L^\infty(\R^n \times [0,\infty))} \leq M$. 
Set $L = \Lam (M+\|f\|_{L^\infty(\R^n)})+1$. 
Fix $T>0$.
For each $\del>0$, we consider the following auxiliary function
\[
\phi(x,y,t) = u(x,t) - u(y,t) - L|x-y| - \del (|y|^2+1)^{1/2} \quad \text{ for } (x,y,t) \in \R^n \times \R^n \times [0,T].
\]
Assume that $\phi$ has a max at $(x_0,y_0,t_0) \in \R^n \times \R^n \times [0,T]$ with $x_0\not=y_0$. We claim that $t_0=0$.
Assume otherwise, then there exists $(\al,p,X) \in P^{2,+} u(x_0,t_0)$ such that
\[
\al -\frac{|p|}{\chi\left(\frac{-\tr\left(b(p)X\right)}{|p|}\right)} \leq f(x_0), 
\]
where $P$ denotes the parabolic semi-jets (see \cite{CIL, G-book} for instance).
Hence,
\[
M+\|f\|_{L^\infty(\R^n)} \geq -\al + f(x_0) \geq \frac{|p|}{\chi\left(\frac{-\tr\left(b(p)X\right)}{|p|}\right)} \geq \frac{|p|}{\Lam}=\frac{L}{\Lam},
\]
which contradicts with the choice of $L$. Hence $t_0 =0$ or $x_0=y_0$. 
We let $\del \to 0$ to get the result with $C=L$.
\end{proof}

\begin{rem}
It is worthwhile to emphasize that if we consider \eqref{eq:truncate} in the two dimensional setting ($n=2$) with $u_0 \equiv 0$, and $f(x)=\mathbf{1}_{\ol{B}(0,R_0)}(x)$ for some $R_0>0$, and all $x\in\R^2$, then 
interestingly, the graph of its maximal solution $u(x,t)$ describes pretty well the shape of Mt.\ Fuji, a stratovolcano. 
Note that \eqref{f-con} does not hold here since $f$ is not continuous. 
We provide a heuristic explanation about a volcano formation model, and explain in details the maximal viscosity solution in this setting in Appendix. 
\end{rem}

\section{Some estimates on asymptotic speed} \label{sec:speed}
In this section, we proceed to study further properties of asymptotic speed for solutions of equations in the previous section. 
Let $u$ be the solution to (C).
Assume (A1)--(A3). 
Let $c_f$ be the asymptotic speed given by \eqref{u-conv}. 
By (A1), we always have 
\begin{equation}\label{ineq:comparison}
c_f\le \max_{x\in\R^n}f(x)=:M_f.  
\end{equation}
We now give further characterization results on $c_f$. 

\subsection{Positive normal velocity ($V>0$)} \label{subsec:V-pos}
We first consider two cases in Subsections \ref{subsec:HJ}, \ref{subsec:volcano}, which are
rather simple because of the fact that normal velocities are always positive.
Indeed, we have the following.

\begin{thm}\label{thm:V-pos}
Assume that \eqref{f-con} holds.
Let $F$ be either the operator given in Subsection {\rm\ref{subsec:HJ}} or Subsection {\rm\ref{subsec:volcano}}. 
Then,  $c_f=M_f$.  
\end{thm}

To prove this theorem, we need the following simple lemma.

\begin{lem}\label{lem:HJ}
Let $f:\R^n \to \R$ be a function satisfying \eqref{f-con}, and $\del>0$ be a given constant.
Let $w$ be the solution to
\begin{equation}\label{eq:w}
\begin{cases}
w_t -\del |Dw| = f(x) \quad &\text{in } \R^n \times (0,\infty),\\
w(x,0)=0 \quad &\text{on } \R^n.
\end{cases}
\end{equation}
Then, 
\[
\lim_{t \to \infty} \frac{w(x,t)}{t} = M_f \quad \text{locally uniformly for } x \in \R^n. 
\]
\end{lem}

\begin{proof}
First of all, it is clear that $\varphi(x,t) =M_ft$ for $(x,t) \in \R^n \times [0,\infty)$ is a supersolution to \eqref{eq:w}.
Therefore, by the usual comparison principle, 
\begin{equation} \label{w-1}
w(x,t) \leq M_ft \quad \text{for all $(x,t) \in \R^n \times [0,\infty)$.}
\end{equation}

Besides, we have the following optimal control formula for $w$
\[
w(x,t)=\sup\left\{ \int_0^t f(\gam(s))\,ds\,:\, \gam \in \AC([0,t],\R^n), \gam(0)=x, |\gam'| \leq \del \text{ a.e. on } [0,t] \right\}.
\]
Here, $\AC([0,t],\R^n)$ is the set of absolutely continuous functions from $[0,t]$ to $\R^n$.
Fix $R>0$ and $x\in B(0,R)$.
Pick $y \in B(0,R_0)$ such that  $f(y)=M_f$. For $t> (R+R_0)/\del$, set
\[
\gam(s)=
\begin{cases}
x+\del s \frac{y-x}{|y-x|} \quad &\text{for } 0 \leq s \leq \frac{|y-x|}{\del},\\
y \quad &\text{for } \frac{|y-x|}{\del} \leq s \leq t.
\end{cases}
\]
Then 
\begin{equation}\label{w-2}
w(x,t) \geq \int_0^t f(\gam(s))\,ds \geq M_f\left(t-  \frac{|y-x|}{\del}\right) \geq M_f \left ( t- \frac{R+R_0}{\del} \right).
\end{equation}
We combine \eqref{w-1} and \eqref{w-2} to reach the conclusion.
\end{proof}

\begin{proof}[Proof of Theorem \ref{thm:V-pos}]
This is a straightforward consequence of Lemma \ref{lem:HJ}.  
We only consider the case in Subsection \ref{subsec:volcano}. 
Let $u$ be the solution to \eqref{eq:truncate}.
Noting that 
$\lam \le g(r)\le \Lam$ for all $r\in\R$, we deduce that $u$ is a supersolution to
\[
u_t-\frac{|Du|}{\Lam}\ge f(x) \quad \text{ in } \R^n \times (0,\infty).
\]
We then use Lemma \ref{lem:HJ} and \eqref{ineq:comparison} to get the conclusion.  
\end{proof}

\subsection{Forced mean curvature flow in the radially symmetric setting} \label{subsec:rad}
In this subsection, we assume $f$ is radially symmetric, that is, 
$f(x) = \tilde f(|x|)$ for $x\in \R^n$, where $\tilde{f}:[0,\infty) \to [0,\infty)$ is given. 
The following theorem gives a complete characterization of $c_f$ in term of $\tilde f$ (or $f$).  
\begin{thm}\label{thm:f-radial}
Assume that $u_0 \in \BUC(\R^n)$ and $f(x) = \tilde f(|x|)$ for $x\in \R^n$,
where $\tilde f\in C_c([0,\infty), [0,\infty)) \cap \Lip([0,\infty), [0,\infty)) $.
Let $u$ be the solution to \eqref{FMCF}. Then
\[
c_f = \max_{r \in [n-1,\infty)} \tilde f(r) = \max_{|x| \geq n-1} f(x).
\]
\end{thm}

In order to prove this theorem, we here consider 
a radially symmetric solution $u(x,t) = \phi(|x|,t)$, 
where $\phi=\phi(r,t):[0,\infty) \times [0,\infty) \to \R$, with 
$u(x,0)=0$ for all $(x,t)\in\R^n\times[0,\infty)$. 
Then, 
\begin{align*}
&u_t=\phi_t, \  
Du=\phi_r \frac{x}{|x|}, \  
D^2u= 
\phi_{rr}\frac{x\otimes x}{|x|^2}
+\phi_r\frac{1}{|x|}\Big(I-\frac{x\otimes x}{|x|^2}\Big). 
\end{align*}
Plugging these into \eqref{FMCF} to reduce it to
\begin{equation}\label{eq:phi}
\begin{cases}
\phi_t - \frac{n-1}{r} \phi_r - |\phi_r| = \tilde f(r) \quad &\text{ in } (0,\infty) \times (0,\infty),\\
\phi(\cdot,0)=0 \quad &\text{ on } [0,\infty).
\end{cases}
\end{equation}
Notice here that since we consider the viscosity solution (which may not be smooth at $x=0$) of \eqref{FMCF}, we do not know the boundary condition of $\phi$ at $r=0$ a priori.

Besides, the Hamiltonian of \eqref{eq:phi} is $H(p,r)=-\frac{n-1}{r} p-|p|-\tilde{f}(r)$ for $(p,r) \in \R \times (0,\infty)$, 
which is concave in $p$ and singular in $r$ at $r=0$.
Its corresponding Lagrangian $L$ is
\[
L(q,r)=
\begin{cases}
\tilde{f}(r) \quad &\text{if} \ \left|q+\frac{n-1}{r}\right| \leq 1, \\
-\infty \quad &\text{otherwise}.
\end{cases}
\]
Let us define the value function $\tilde{\phi}: (0,\infty) \times [0,\infty)$ with a state constraint condition by 
\begin{equation}\label{rep-phi}
\tilde{\phi}(r,t)=\sup \left\{ \int_0^t \tilde{f}(\gam(s))\,ds\,:\,
\gam([0,t]) \subset (0,\infty), \ \gam(t)=r,\ \left|\gam'(s)+\frac{n-1}{\gam(s)}\right| \leq 1 \ \text{a.e.} \right\}.
\end{equation}

\begin{lem}\label{lem:regularity}
Let $\tilde{\phi}(r,t):(0,\infty)\times[0,\infty)$ be the function defined by \eqref{rep-phi}. 
Then, $\tilde{\phi}$ is Lipschitz continuous on $(0,\infty)\times[0,T]$ for any $T>0$, 
and is a viscosity solution to \eqref{eq:phi}. 
\end{lem}

\begin{proof}
Let $0<r_1<r_2$ and $t>0$. 
We first consider the case where $n-1<r_1<r_2$ with $r_2-r_1<t$. 
Take an arbitrary $\gam$ in the admissible class of \eqref{rep-phi} 
such that $\gam(t)=r_2$.
Let $\eta$ be the solution of the following ODE
\[
\begin{cases}
\displaystyle
\eta'(s) +\frac{n-1}{\eta(s)}=-1  \quad \text{ for } s>0,\\
\eta(0) = r_2.
\end{cases}
\]
Since $|\eta'(s)|=1+(n-1)/\eta(s)\ge 1$ as long as $\eta(s)>0$, there exists $\alpha_1>0$ such that 
\[
\eta(\alpha_1)=r_1, \quad
\alpha_1\le r_2-r_1<t. 
\]

Set $\tilde \gam: [0,t] \to (0,\infty)$ such that
\[
\tilde \gam (s)=
\begin{cases}
\displaystyle
\gam(s+\alpha_1) \quad &\text{ for } 0\leq s \leq t-\alpha_1,\\
\eta (s- (t-\alpha_1)) \quad &\text{ for } t -\alpha_1 \leq s \leq t.
\end{cases}
\]
Then $\tilde \gam$ is also in the admissible class of \eqref{rep-phi} with $\tilde \gam(t) = r_1$.
Because of the boundedness of $f$, one has 
\begin{align*}
\tilde{\phi}(r_1,t)\ge&\, 
\int_0^t \tilde f(\tilde{\gam}(s))\,ds
=\int_0^{t-\alpha_1} \tilde f(\gam(s+\alpha_1))\,ds+\int_{t-\alpha_1}^t \tilde f(\tilde{\gam}(s))\,ds\\
\ge&\,  
\int_{\alpha_1}^t \tilde f(\gam(s))\,ds-C\alpha_1
\ge  \int_0^t \tilde f(\gam(s))\,ds-C'\alpha_1. 
\end{align*}
Take the supremum of the above over all admissible curves $\gam$ to yield
\[
\tilde{\phi}(r_2,t) \leq \tilde{\phi}(r_1,t) + C \alpha_1 \le \tilde{\phi}(r_1,t) + C (r_2-r_1).
\]
By a similar argument, we get 
\[
\tilde{\phi}(r_1,t) \leq \tilde{\phi}(r_2,t) + C (r_2-r_1).
\]

We next consider the case where 
$r_1<r_2\le n-1$ with $r_2-r_1\le \beta$, where 
$\beta$ will be fixed later. 
We repeat the above argument with a slight modification. 
Take an arbitrary $\gam$ in the admissible class of \eqref{rep-phi} such that $\gam(t)=r_2$.
Let $\eta$ be the solution of the following ODE
\[
\begin{cases}
\displaystyle
\eta'(s) = -\frac{n-1}{\eta(s)} \quad \text{ for } s>0,\\
\eta(0) = r_2.
\end{cases}
\]
Noting that $|\eta'|=(n-1)/\eta\ge (n-1)/r_2$, we see that there exists 
$\alpha_2>0$ such that $\eta(\alpha) = r_1$ and 
\[
\alpha_2 \leq \frac{r_2 (r_2-r_1)}{n-1}.
\]
Choose $\beta>0$ so small that  $r_2 (r_2-r_1)/(n-1)<t$. 

By a similar argument to the above, we obtain 
\begin{equation}\label{ineq:near0}
|\tilde{\phi}(r_1,t)-\tilde{\phi}(r_2,t)| \leq  C r_2(r_2-r_1). 
\end{equation}

Similarly, we can prove the Lipschitz continuity with respect $t$, and 
we obtain the conclusion. 

By using the dynamic programing principle, we can easily prove that 
$\tilde{\phi}$ is a viscosity solution to \eqref{eq:phi}. 
\end{proof}

In view of Lemma \ref{lem:regularity}, the function $\tilde{\phi}$ can be uniquely 
extended to a continuous function on $(r,t)\in[0,\infty)\times[0,T]$. 
We still denote it by $\tilde{\phi}$. 

\begin{lem} \label{lem:phi-r-0}
Let $\tilde{\phi}$ be the function on $[0,\infty)\times[0,\infty)$ defined in the above. 
Then, $\tilde{\phi}_r(0,t)=0$ for all $t>0$.
\end{lem}
The proof of this lemma is a straightforward result of inequality \eqref{ineq:near0}. 

\begin{lem}
Set $u(x,t):=\tilde{\phi}(|x|,t)$ for all $(x,t) \in \R^n \times [0,\infty)$. Then, $u$ is the viscosity solution to \eqref{FMCF}.  
\end{lem}
\begin{proof}
It is clear from Lemma \ref{lem:regularity} and \cite[Lemma A.1, Appendix A]{GMT} that 
$u$ is a viscosity solution to
\begin{equation}\label{temp}
\begin{cases}
u_t -\left(\Div\left(\frac{Du}{|Du|}\right)+1\right)|Du|=f(x) \quad &\text{in} \ (\R^n \setminus \{0\}) \times(0,\infty),\\
u(0,t) = \tilde{\phi}(0,t) \quad &\text{in } (0,\infty),\\
u(x,0)=u_0(x) \quad &\text{on } \R^n.
\end{cases}
\end{equation}
We thus only need to check that $u$ is a viscosity solution to \eqref{FMCF} at $x=0$.
Note first that, in light of Lemma \ref{lem:phi-r-0}, $Du(0,t)=0$ for all $t>0$.
Let us only check the viscosity subsolution property at $x=0$ as the supersolution property at $x=0$ follows in a similar manner.

Let $\varphi$ be a  smooth test function such that $u - \varphi$ has a strict maximum at $(0,t_0)$ for some $t_0>0$.
Obviously, $D\varphi(0,t_0)= Du(0,t_0)=0$.
Let $\{p_k\} \subset \R^n$ be a sequence of non-zero vectors such that $|p_k|$ is sufficiently small for all $k\in \N$ and $\lim_{k \to 0} p_k=0$.
For each $k \in \N$, we have that $u(x,t) - \varphi(x,t) - p_k\cdot x$ attains a local maximum at $(x_k, t_k)$ and, by passing a subsequence 
if necessary, $\lim_{k \to \infty} (x_k, t_k) = (0,t_0)$.
Since $p_k \neq 0$, $x_k \neq 0$ for all $k \in \N$. Set $q_k = D\varphi(x_k,t_k) + p_k$ for all $k \in \N$.
By the definition of the viscosity subsolution, we yield
\[
\varphi_t(x_k,t_k) - \tr \left(\left(I_n - \frac{q_k \otimes q_k}{|q_k|^2}\right) D^2 \varphi(x_k,t_k)\right) - |q_k| \leq f(x_k).
\]
Let $k\to \infty$ to get the desired conclusion.
\end{proof}

We are now ready to prove the main result in this subsection, Theorem {\rm\ref{thm:f-radial}}.

\begin{proof}[Proof of Theorem {\rm\ref{thm:f-radial}}]

Take $R \in (0,n-1)$. If $\gam$ is in the admissible class of \eqref{rep-phi} such that $\gam(s) \in (0,R)$, then
\[
\gam'(s) \leq 1 - \frac{n-1}{\gam(s)} \leq 1 - \frac{n-1}{R} = - \frac{n-1-R}{R}=:-d<0.
\]
Hence
\[
\left| \{s\in [0,t]\,:\, \gam(s) \in (0,R)\} \right| \leq \frac{R}{d}.
\]
Here, for a Lebesgue measurable set $A$, $|A|$ denotes its Lebesgue measure. 
In particular, for $t>R/d$, we have that
\begin{equation}\label{phi-1}
\phi(r,t) \leq \frac{R}{d} \max_{r \in [0,\infty)} \tilde f(r) + \left(t-\frac{R}{d}\right) \max_{r \geq R} \tilde f(r).
\end{equation}
Divide both sides of \eqref{phi-1} by $t$ and let $t \to \infty$ to yield
\begin{equation*}
c_f \leq \max_{r \geq R} \tilde f(r).
\end{equation*}
We then let $R \to n-1$ to get
\begin{equation}\label{phi-2}
c_f \leq \max_{r\geq n-1} \tilde f(r).
\end{equation}

We need to show that the reverse inequality of \eqref{phi-2} holds as well. 
Take $\ol{r}_0 \geq n-1$ such that $\tilde f(\ol{r}_0)=\max_{r \geq n-1} \tilde{f}(r)$. 
Fix $r \in (0,\infty)$ and $r_0>\ol{r}_0$, and consider two cases.

\smallskip
{\it Case 1: $r>r_0$.} 
Set $T_1:=\frac{r_0(r-r_0)}{r_0-(n-1)}\in(0,\infty)$. For $t>T_1$, define $\gam:[0,t] \to (0,\infty)$ as 
\[
\gam(s):=
\begin{cases}
r_0 \quad &\text{for} \ 0<s<t-T_1,\\
\displaystyle
r_0+(s-t+T_1)\frac{r_0-(n-1)}{r_0} \quad &\text{for} \ t-T_1<s<t.
\end{cases}
\]
It is clear that $\gam$ is admissible in formula \eqref{rep-phi} and hence
\[
\phi(r,t) \geq \int_0^t \tilde{f}(\gam(s))\,ds \geq
\int_0^{t-T_1} \tilde f(\gam(s))\,ds  =\tilde f(r_0) (t-T_1),
\]
as $\tilde{f}$ is nonnegative, which is sufficient to get the conclusion by letting $r_0 \to \ol{r}_0$.

\smallskip

{\it Case 2: $0 < r \leq r_0$.} 
We first consider the following ODE
\[
\begin{cases}
\displaystyle
\xi'(s)=-1-\frac{n-1}{\xi(s)} \quad &\text{for}\ s>0,\\
\xi(0)=r_0.
\end{cases}
\]
Take $T_2>0$ to be the smallest value such that $\xi(T_2)=r$. It is immediate that $T_2 \leq r_0-r$.
For $t>T_2$, consider  $\gam:[0,t] \to (0,\infty)$ as 
\[
\gam(s):=
\begin{cases}
r_0 \quad &\text{for} \ 0\le s<t-T_2,\\
\xi(s-t+T_2) \quad &\text{for} \ t-T_2<s\le t.
\end{cases}
\]
Again,  it is obvious that $\gam$ is admissible in formula \eqref{rep-phi} and 
\[
\phi(r,t) \geq \int_0^t \tilde f(\gam(s))\,ds \geq \tilde f(r_0) (t-T_2).
\]
The proof is complete by letting $r_0 \to \ol{r}_0$. 
\end{proof}

\subsection{Forced mean curvature flow in non-radially symmetric settings} \label{subsec:non-rad}
In non-radially symmetric settings, the situation seems much more complicated.
 At least at this moment, it is quite hard to obtain detailed qualitative properties of the asymptotic speed. 
We here give some partial results based on the analysis of the radially symmetric setting in Subsection \ref{subsec:rad}. 


The first result concerns a situation where $f$ does not take values near its maximum outside of the critical ball $B(0,n-1)$.
\begin{lem}\label{lem:c-less-M}
Assume that $u_0 \in \BUC(\R^n)$ and $f:\R^n \to \R$ satisfying \eqref{f-con}.  Assume further that there exist $s \in (0,M_f)$ and  $R < n-1$ such that
\[
 \{x\in \R^n\,:\,M_f-s \leq f(x) \leq M_f\} \subset B(0,R).
\]
Let $u$ be the solution to \eqref{FMCF}. 
Then $c_f \leq M-s$.
\end{lem}

\begin{proof}
Define
\[
\ol{f}(x) = \max_{|y|=|x|} f(y) \quad \text{ for } x\in \R^n.
\]
Then $\ol{f}\geq f$,  $\ol{f}$ is radially symmetric, and $\{\ol{f}\ge M_f-s\}\subset B(0,n-1)$ by assumption. In particular, 
\[
 \max_{|x| \geq n-1} \ol{f}(x)\le M_f-s.
 \]
 Let $v$ be the solution to \eqref{FMCF} with the right hand side $\ol{f}$ in place of $f$.
 Then, by the comparison principle, $0 \leq u \leq v$. This, together with Theorem \ref{thm:f-radial}, implies
 \[
 0 \leq c_f \leq M_f-s.
 \qedhere
 \]
\end{proof}

Next, we consider a setting where $f$ takes its maximum value in the whole critical ball $B(0,n-1)$,
in which case we easily verify that $c_f=M_f$.

\begin{lem}\label{lem:c-M}
Assume that $u_0 \in \BUC(\R^n)$ and $f:\R^n \to \R$ satisfying \eqref{f-con}.  
Assume further that there exists $R \geq n-1$ such that
\[
B(0,R) \subset \{x\in \R^n\,:\,f(x)=M_f\}
\]
Let $u$ be the solution to \eqref{FMCF}. 
Then $c_f=M_f$.
\end{lem}

\begin{proof}
The proof goes in a similar manner to that of Lemma \ref{lem:c-less-M}.
Define
\[
\ul{f}(x) = \min_{|y|=|x|} f(y) \quad \text{ for } x\in \R^n.
\]
Then, $\ul{f} \leq f$, $\ul{f}$ is radially symmetric, $\ul{f}=M_f$ on $\ol{B}(0,R)$. In particular,
\[
\max_{|x| \geq n-1} \ul{f}(x)=M_f.
\]
Let $w$ be the solution to \eqref{FMCF} with the right hand side $\ul{f}$ in place of $f$.
 Then, by the comparison principle, $0 \leq w \leq u$. This, together with Theorem \ref{thm:f-radial}, implies
 \[
M_f= c_{\ul{f}} \leq c_f \leq M_f.
 \qedhere
 \]
\end{proof}

\begin{rem}\label{rem:large-time}
Under the setting of Lemma \ref{lem:c-M}, $c_f=M_f$.
It is important going further to study finer asymptotics of $u(x,t)$ as $t \to \infty$ (more or less next terms in the asymptotic expansion of $u$).
A natural question to ask here is 
\[
\lim_{t \to \infty} \left( u(x,t) - c_f t \right) = ?
\]
In general, this is an open problem as we are in the setting that 
$F$ is fully nonlinear, and degenerate elliptic (thus no strong maximum principle)
and $F$ is not convex in $p$. See discussions in \cite[Section 5.7]{LMT}.
We will address this question in the near future.
\end{rem}

\begin{rem}
The asymptotic limit defined in Remark \ref{rem:large-time} is sometimes 
called (unrescaled) asymptotic profile. In the case of forced mean 
curvature flow, it is known in \cite[Theorem 1.4]{Giga-ICM} that a rescaled 
asymptotic profile is of the form 
\[
\lim_{\lambda \rightarrow \infty} u(\lambda x,\lambda t)/\lambda = 
c_f(t-|x|)_{+}
\]
if  $u_0=0$; as far as $u_0$ is compactly supported it can be easily 
generalized for the case $u_0$ is not identically equal to zero. 
In \cite{Giga-ICM} this rescaled limit is also established when the spreading law 
is anisotropic. 

Note that in \cite{GMT} and also in \cite{Giga-ICM}, it is shown that 
the asymptotic speed $c_f$ can be strictly smaller than $M_f$.  
In \cite{GMT} the case when $f$ is a characteristic function is discussed in detail. 
By the way, it is shown that the maximum points of $u(\cdot,t)$ is contained in 
the convex hull of the set of maximum points of $f$ in \cite[Lemma 3.7]{Giga-ICM},  
which is stronger than (A3). 
\end{rem}

Finally, we recall some front propagation problems with obstacles developed in \cite{GMT}.
For an open set $A \subset \R^n$ (resp., a closed set $B\subset\R^n$), 
we denote by $\cF^-[A](t)$ (resp., $\cF^+[B](t)$) 
the level set solution of the following front propagation with obstacles
\begin{align*}
&V=\kappa+1\quad \textrm{ with obstacle }  A, \textrm{ i.e., }  \cF^-[A](t)\subset A\\
&{\rm(resp.,}\ V=\kappa+1\quad \textrm{ with obstacle }  B, \textrm{ i.e., }  \ B  \subset \cF^+[B](t){\rm)},
\end{align*} 
for any $t\ge0$, and $\cF^-[A](0)=A$, $\cF^+[B](0)=B$. 

\begin{lem}\label{lem:obstacle}
Assume that $u_0 \in \BUC(\R^n)$ and $f:\R^n \to \R$ satisfying \eqref{f-con}.  
Let $u$ be the solution to \eqref{FMCF}. We have the following conclusions.
\begin{itemize}
\item[{\rm(i)}]  
If there exists $t_0>0$ such that
\[
\cF^{-}[\{f>0\}](t_0)=\emptyset,
\]
then $c_f <M_f$.
\item[{\rm(ii)}]  
If there exists $s\in (0,M_f)$ such that
\[
\cF^{+}[\{f \geq s\}](t)\to\R^n \quad\text{as} \ t\to\infty,
\]
then $c_f >0$.
\end{itemize}
\end{lem}

This is a straightforward result of Theorem \ref{thm:main} together with 
\cite[Theorems 5.4, 5.6]{GMT}.


\section{Further on asymptotic speed and numerical results}\label{sec:numerical}

In this section, we give numerical schemes for the birth and spread type PDEs, \eqref{FMCF} and \eqref{eq:truncate}, and provide numerical results on asymptotic speed of \eqref{FMCF}  in two dimensions ($n=2$). 
With the aid of numerical simulation, we raise several concrete questions to be studied in the future.

\subsection{Numerical Schemes}

We discretize \eqref{FMCF} and \eqref{eq:truncate} 
by the usual finite difference schemes.
We now recall the discretization
of the curvature term as in \cite{OTG}
with some remarks for equations with outer force term.
See also \cite{OF}.

We discretize the spatial derivative terms 
of the equations 
on the Cartesian grid
\[
 D = \{ x_{i,j} = (i \Delta x, j  \Delta x) \,:\, \ -N \le i,j \le N \}
\]
of a square domain 
$\Omega = [-R,R]^2 \subset \mathbb{R}^2$ 
with a uniform grid spacing $\Delta x > 0$ and number of
points $N \in \mathbb{N}$.
We first omit the time variable $t$
for simplicity to obtain the discretization
of the curvature and eikonal term.
For a given function $u:  \Omega \to \R$, let us set 
$u_{i,j} = u (i \Delta x, j \Delta x)$.
To avoid the division by zero on the curvature term,
we introduce a regularized curvature term
\[
 \tilde{\kappa} = \mathrm{div}\left(\frac{D u}{\sqrt{\varepsilon^2 + |D u|^2}}\right).
\]
Then, the forced mean curvature operator with source term in \eqref{FMCF} is approximated as
\begin{equation}
 \label{spatial diff of FMCF}
  \Phi (u)_{i,j}
  := |\hat{D} u_{i,j}| \tilde{\kappa}_{i,j} + |\tilde{D} u_{i,j}| + f(x_{i,j}). 
\end{equation}
The curvature term $\tilde{\kappa}_{i,j}$ is discretized as follows
\begin{align*}
 & \tilde{\kappa}_{i,j} = \frac{P^+_{i,j} - P^-_{i,j}}{\Delta x} 
 + \frac{Q^+_{i,j} - Q^-_{i,j}}{\Delta x}, \\
 & P^\pm_{i,j} = \frac{\partial^\pm_{x_1} u_{i,j}}{\sqrt{\varepsilon^2 + (\partial^\pm_{x_1} u_{i,j})^2
 + (\bar{\partial}^\pm_{x_2} u_{i,j})^2}}, \quad 
 Q^\pm_{i,j} = \frac{\partial^\pm_{x_2} u_{i,j}}{\sqrt{\varepsilon^2
 + (\bar{\partial}^\pm_{x_1} u_{i,j})^2 + (\partial^\pm_{x_2} u_{i,j})^2}},
\end{align*}
where the partial differences $\partial^\pm_{x_1} u_{i,j}$ and
$\bar{\partial}^\pm_{x_1} u_{i,j}$ on $x_1$ are given by
\begin{align*}
 \partial^\pm_{x_1} u_{i,j} 
 = \pm \frac{u_{i \pm 1,j} - u_{i,j}}{\Delta x}, \quad
 \bar{\partial}^\pm_{x_1} u_{i,j} = \frac{1}{2\Delta x}
 \left( \frac{u_{i + 1, j \pm 1} + u_{i+1,j}}{2} - \frac{u_{i - 1, j \pm 1} + u_{i-1,j}}{2} \right).
\end{align*}
The terms $\partial^\pm_{x_2} u_{i,j}$ and $\bar{\partial}^\pm_{x_2} u_{i,j}$
are also defined analogously as above.
The term $|\hat{D} u_{i,j}|$ in front of $\tilde{\kappa}$ is discretized by
\begin{align*}
 |\hat{D} u_{i,j}| & = \sqrt{|\hat{\partial}_{x_1} u_{i,j}|^2 + |\hat{\partial}_{x_2} u_{i,j}|^2}, \\
 |\hat{\partial}_{x_1} u_{i,j}| & =
 \left\{
 \begin{array}{ll}
  {\displaystyle \left| \frac{u_{i+1,j} - u_{i-1,j}}{2} \right|} &
   {\displaystyle \mbox{if} \ \left| \frac{u_{i+1,j} - u_{i-1,j}}{2} \right| \ge \rho,} \\
  \max \{ |\partial^+_{x_1} u_{i,j}|, | \partial^-_{x_1} u_{i,j}| \}
   & \mbox{otherwise}
 \end{array}
 \right.
\end{align*}
with a small constant $\rho > 0$.
The term $|\hat{\partial}_{x_2} u_{i,j}|$ is defined as the same manner of
$|\hat{\partial}_{x_1} u_{i,j}|$.
We remark that $\rho > 0$ should be chosen small, but not too small.
In fact, for example, when $u(x) = - |x|^2/2$ then it is well known that
$|D u|\mathrm{div} (D u/|D u|) = - 1$ in viscosity sense.
However, if $|D u|$ was approximated just with the center differences
$(u_{i+1,j} - u_{i-1,j})/(2 \Delta x)$ 
and $(u_{i,j+1} - u_{i,j-1})/(2 \Delta x)$, then
the numerical result of $|D u_{0,0}| \tilde{\kappa}_{0,0}$ would be zero.
This discrepancy would cause some irregular numerical results,
in particular, when $f \not\equiv 0$.
We choose adequate $\rho > 0$
to avoid such irregular numerical results.

On the other hand, the first order term
$|\tilde{D} u_{i,j}| = \sqrt{|\tilde{\partial}_{x_1} u_{i,j}|^2 + |\tilde{\partial}_{x_2} u_{i,j}|^2}$ is
discretized with an upwind differencing 
\begin{align*}
 |\tilde{\partial}_{x_1} u_{i,j}| = \max \{ (\tilde{\partial}^+_{x_1} u_{i,j})_+, (-\tilde{\partial}^-_{x_1} u_{i,j})_+ \},
\end{align*}
where $(a)_+ = \max \{ a, 0 \}$ for $a \in \mathbb{R}$,
\begin{align*}
 \tilde{\partial}^\pm_{x_1} u_{i,j} & = \partial^\pm_{x_1} u_{i,j}
 \mp \frac{\Delta x}{2}
 \mu \left( \frac{u_{i \pm 2, j} - 2 u_{i \pm 1, j} + u_{i,j}}{{\color{red} \Delta x}^2},
 \frac{u_{i+1,j} - 2u_{i,j} + u_{i-1,j}}{ \Delta x^2} \right), \\
 \mu (p,q) & =
 \left\{
 \begin{array}{ll}
  p & \mbox{if} \ |p| < q, \\
  q & \mbox{otherwise}.
 \end{array}
 \right.
\end{align*}
The term $|\tilde{\partial}_{x_2} u_{i,j}|$ is also defined
with the same manner as $\tilde{\partial}_{x_1} u_{i,j}$.

\smallskip

We now let $u:\Om \times [0,\infty) \to \R$ be the unknown.
We calculate the approximate force mean curvature flow
\[
u_t =   \Phi (u)_{i,j} = |\hat{D} u_{i,j}| \tilde{\kappa}_{i,j} + |\tilde{D} u_{i,j}| + f(x_{i,j})
\]
with an explicit finite difference scheme,
i.e.,
\begin{equation}
 \label{diff FMCF}
 u^{k}_{i,j} = u^{k-1}_{i,j} + \Delta t \Phi(u)^{k-1}_{i,j}, 
\end{equation}
where $u^k_{i,j} = u(i \Delta x, j \Delta x, k \Delta t)$ with a time span $\Delta t > 0$ for $k\in \N$.

\smallskip

Next, we consider a discretized equation of a truncated inverse mean curvature flow equation with a source \eqref{eq:truncate}. 
Notice that \eqref{surface:t-inv} with \eqref{func:truncate}
has a direction of the flow
such that $V > 0$. 
Hence, $V = u_t / |Du|$ for \eqref{surface:t-inv}
should be discretized with an upwind
differencing like as the eikonal equation.
In this paper we choose the following scheme for \eqref{eq:truncate}
\begin{equation}
 \label{discretized tr-eq}
 u^{k}_{i,j}
 = u^{k-1}_{i,j} +
 \Delta t \left( \frac{|\bar{D} u^{k-1}_{i,j}|}{\chi(- \tilde{\kappa}^{k-1}_{i,j})}
 + f_{i,j} \right),
\end{equation}
where $\chi$ is the function defined by \eqref{func:truncate}. 
Note that the coefficient $|\bar{D} u_{i,j}^{k-1}|$
in \eqref{discretized tr-eq} is calculated with
\begin{align*}
 |\bar{D} u^{k-1}_{i,j}|
 & = \sqrt{|\bar{\partial}_{x_1} u^{k-1}_{i,j}|^2 
 + |\bar{\partial}_{x_2} u^{k-1}_{i,j}|^2}, \\
 | \bar{\partial}_{x_1} u^{k-1}_{i,j} | 
 & = \max \{ (\partial_{x_1}^+ u^{k-1}_{i,j})_+, 
 - (\partial_{x_1}^- u^{k-1}_{i,j})_- \},
\end{align*}
and $|\bar{\partial}_{x_2} u^{k-1}_{i,j}|$
is calculated with the same manner as that of
$|\bar{\partial}_{x_1} u^{k-1}_{i,j}|$.
Although the approximation order of $|\bar{D} u|$
is lower than that of $|\tilde{D} u|$,
$|\bar{D} u|$ is more accurate 
than $|\tilde{D} u|$ for \eqref{discretized tr-eq}
because of the direction of the flow.

Note that the above discretization implies 
the data outside of the domain, that is, 
$u^k_{\pm (N+1), j}$ or $u^k_{i, \pm (N+1)}$
for $-(N+1) \le i,j \le N+1$.
These data should be given by a boundary condition.
Although such a situation is different from 
the Cauchy problem we considered in the previous sections,
we impose the Neumann boundary condition
$\vec{\nu} \cdot D u = 0$,
i.e., 
\begin{align*}
 & u_{\pm (N+1), j} = u_{\pm N, j}, \quad
 u_{i, \pm (N+1)} = u_{i, \pm N} \quad \mbox{for} \ -(N+1) \le i,j \le N+1
\end{align*}
for the numerical simulations in this paper.


\subsection{Propagation by forced mean curvature and asymptotic speed}

In this subsection, we present some numerical results
of asymptotic speed $c = \lim_{t \to \infty} u(x,t)/t$
on \eqref{FMCF} 
by \eqref{diff FMCF} with \eqref{spatial diff of FMCF}.
We already have some mathematical results on the asymptotic speed
for several concrete examples of $f(x)$.
We now verify them, and give some numerical 
predictions on the asymptotic speed for other cases.

Throughout this subsection, we set
$R = 2.56$ and $N=128$, then $\Delta x = 0.02$
for parameters of the spatial discretization.
The time span for \eqref{diff FMCF}
is chosen as $\Delta t = 0.2 \times \Delta x^2$.
For the solution $u(x,t)$ to \eqref{FMCF}, 
the numerical data of the growth speed of $u$ 
at time $t$
is calculated by
the mean value of $u(x,t)/t$ on $\Om$, that is, 
\begin{equation}
 \label{num asym speed}
  c_\triangle (t)
  := \frac{1}{|\Omega|} \int_{\Omega} \frac{u(x,t)}{t} \,dx
  = \frac{1}{(2R)^2} \int_{[-R,R]^2} \frac{u(x,t)}{t} \,dx, 
\end{equation}
where $|\Omega|$ is the Lebesgue measure of $\Omega$.
Then, we see the numerical results of the asymptotic speed
$\lim_{t \to \infty} u(x,t)/t$ by
$c_\triangle (T)$ with $T>0$ chosen large enough.
In this section we choose $T = 40$
so that we calculate \eqref{diff FMCF}
for $1 \le k \le K = 5 \times 10^5$.
The parameters $\varepsilon$ and $\rho$ are chosen as
$\varepsilon = 0.001$, and $\rho = 0.01$, respectively.

\par
\bigskip
\noindent
\textbf{Example 1. A downward cone source (benchmark test).}
We first consider the case that
\begin{equation}
 \label{num ex1}
 f(x) = (r - |x|)_+  \quad \text{ for } x\in \R^2,
\end{equation}
with $r \ge 0$ given,
where $|x|$ is the usual norm of $x \in \mathbb{R}^2$.
In this case, we have 
\begin{equation}
 \label{asym spd ex1}
 c = c(r) = (r-1)_+. 
\end{equation}
due to Theorem \ref{thm:f-radial}.
From the above result, 
the numerical data also depends on $r$,
i.e., $c_\triangle = c_\triangle (t;r)$. 
The left figure of Figure \ref{force1: r-rate}
presents a graph of $r \mapsto c_\triangle (40;r)$ 
and $r \mapsto c(r)$ (dashed line) for this situation.
\begin{figure}[htbp]
 \begin{center}
  \includegraphics[scale=0.8]{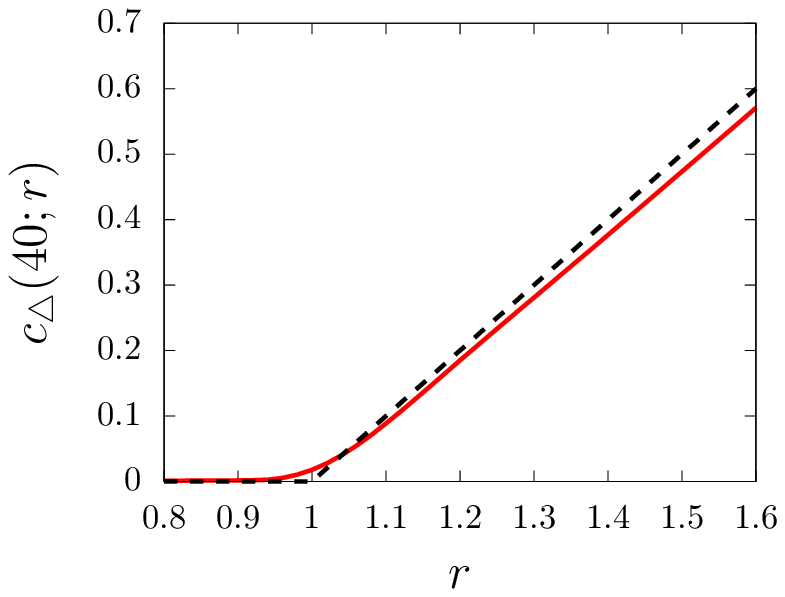}
  \includegraphics[scale=0.8]{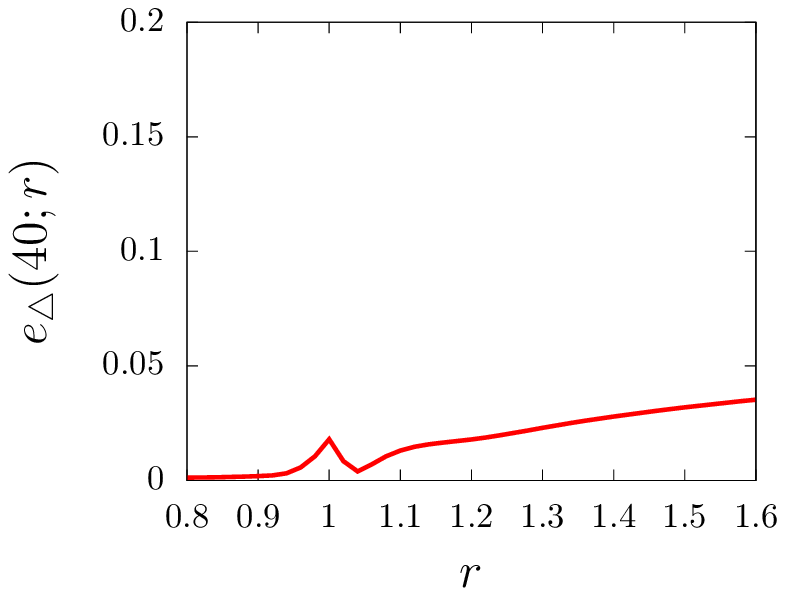}
  \caption{Graph of the numerical results
  $r \mapsto c_\triangle (t;r)$ (left)
  and the $L^2$ error
  $r \mapsto e_\triangle (t;r) := \| u (\cdot,t)/t - c(r) \|_{L^2}/(4R^2)$
  (right)
  at $t = 40$
  for the case \eqref{num ex1}.
  The dashed line in the left figure denotes $r \mapsto c(r)$.}
  \label{force1: r-rate}
 \end{center}
\end{figure}

In this case, we have a rigorous target as in \eqref{asym spd ex1}
so that we calculate the $L^2$ error
\[
 e_\triangle (t;r) = \frac{1}{(2R)^2}
 \left\| \frac{u(\cdot, t)}{t} - c(r) \right\|_{L^2}.
\]
The right figure of Figure \ref{force1: r-rate}
presents a graph of $r \mapsto e_\triangle (40;r)$,
from which one can find $c_\triangle (\cdot;r)$ is
very close to $c(r)$.
Figure \ref{force1: time-rate} shows the graphs
of $t \mapsto c_\triangle (t;r)$ for $r = 0.80$, $1.00$, $1.20$,
$1.40$ and $1.60$, which implies
$c_\triangle (t;r)$ seems to be converging to a 
constant monotonically for $t \gg 1$.
\begin{figure}[!htbp]
 \begin{center}
  \includegraphics[scale=0.8]{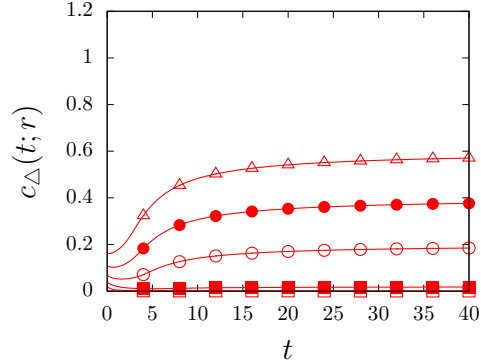}
  \caption{Graph of $t \mapsto c_\triangle (t;r)$
  for $r = 0.80$ (the line with $\square$),
  $1.00$ (with $\blacksquare$),
  $1.20$ (with $\circ$),
  $1.40$ (with $\bullet$)
  and $1.60$ (with $\triangle$).}
  \label{force1: time-rate}
 \end{center}
\end{figure}

Note that the critical value of $r$,
which is the maximum of $r$ satisfying $c_\triangle (t;r) \approx 0$,
is slightly different from $1.0$.
It is not because of the finite terminal time $T = 40$,
but the numerical error by
the discretization of \eqref{FMCF}
and approximation of the curvature term.
See Figure \ref{force1: time-rate_crit} which are
graphs of $t \mapsto c_\triangle (t;r)$
focusing up the results around $r = 1.0$.
One can find $t \mapsto c_\triangle (t;r)$
is clearly increasing for $t \gg 1$ if $r \ge 0.98$.
More precisely, our numerical data shows that
$t \mapsto c_\triangle (t;r)$
increases if $t \ge 5.04$ for $r = 1.00$,
$t \ge 7.16$ for $r = 0.98$,
and $t \ge 8.96$ for $r = 0.96$.
\begin{figure}[!htbp]
 \begin{center}
  \includegraphics[scale=0.8]{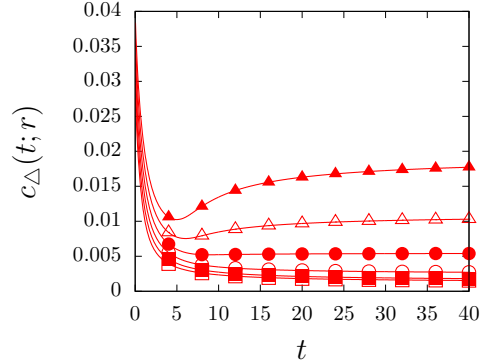}
  \caption{Graph of $t \mapsto c(t;r)$
  for $r = 0.90$ (the line with $\square$),
  $0.92$ (with $\blacksquare$),
  $0.94$ (with $\circ$),
  $0.96$ (with $\bullet$),
  $0.98$ (with $\triangle$),
  and $1.00$ (with $\blacktriangle$).
  The case $r \ge 0.98$ clearly shows the trend aforementioned.}
  \label{force1: time-rate_crit}
 \end{center}
\end{figure}


\par
\bigskip
\noindent
\textbf{Example 2. A square downward cone.}
We next examine the $\ell^1$-norm type force term 
\begin{equation}
 \label{num ex2}
  f(x) = \left(r - |x|_1 \right)_+ = \left( r - (|x_1| + |x_2|) \right)_+
  \quad \mbox{for} \ x = (x_1, x_2) \in \mathbb{R}^2.
\end{equation}
A discontinuous case of this example was first studied in \cite{GMT}. 
Also, see an interesting paper \cite{LD} of the crystal growth. 

In this case, we analytically have
\begin{numcases}
{c(r) = \lim_{t \to \infty} c(t;r)=} 
  = 0 & if  \ $r < 1$, \label{c1}\\
  > 0 & if \ $1 < r < \sqrt{2}$, \label{c2}\\
  \geq r - \sqrt{2} & if \ $r > \sqrt{2}$. \label{c3}
\end{numcases}
Indeed, by constructing a source function $g$ which is radially symmetric, 
and satisfies $f(x)\le g(x)$ for $x\in\R^2$, and using Theorem \ref{thm:f-radial} and the comparison principle, 
we obtain \eqref{c1}. 
Also, by Lemma \ref{lem:obstacle} and \cite[Section 6]{GMT}, we obtain 
\eqref{c2}. 
Finally, noting that 
\[
f(r)\ge (r-\sqrt{2}-\ep)\mathbf{1}_{\ol{B}(0,1+\ep)}, 
\]
for any $\ep\in(0,1)$,  
and using Theorem \ref{thm:f-radial} and the comparison principle again, 
we obtain inequality \eqref{c3}.  

As we see in the above, we only have quite partial information on $c(r)$. 
We want to understand the profile of $c(r)$ more, and with the aid 
of numerical simulation, we raise some concrete questions below.

Figure \ref{force2:rad-rate} presents a graph of
$r \mapsto c_\triangle (40;r)$ for $0.8 \le r \le 2.0$.
The chain line in the graph denotes 
the lower bound line $r \mapsto r - \sqrt{2}$.
One can find that $c_\triangle (40;r)$ looks like a line
for $1.6 \le r \le 2.0$.
We now calculate the fitting line of
$c_\triangle (40;r)$ by the least square method
with the data for $1.6 \le r \le 2.0$
and obtain
\[
 c_\triangle (40;r) \approx 0.966253713 \times r -1.238608703,
\]
which is drawn as the dashed line in Figure \ref{force2:rad-rate}. 

\medskip
\textbf{Question 1.} 
From the observation of Figure \ref{force2:rad-rate}, it is reasonable to raise 
the following questions: 
\begin{itemize}
\item[(i)] Are there $r_0>\sqrt{2}$ and $a<\sqrt{2}$ such that $c(r)=r-a$ for all $r\ge r_0$? 
\item[(ii)] Is the function $r\mapsto c(r)$ convex? 
\end{itemize}

\begin{figure}[!htbp]
 \begin{center}
  \includegraphics[scale=0.8]{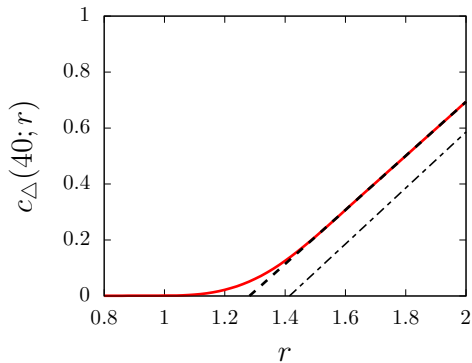}
 \end{center}
 \caption{Graph of $r \mapsto c_\triangle (40;r)$ 
 for \eqref{num ex2}.
 The dashed line means the fitting line
 of $c_\triangle (40;r)$ calculated with the data
 for $1.6 \le r \le 2.0$.
 The chain line means the lower bound of $c(r)$,
 which is $r \mapsto r - \sqrt{2}$.}
 \label{force2:rad-rate}
\end{figure}



\par
\bigskip
\noindent
\textbf{Example 3. Sum of two downward cones.}
We next consider the situation that the source term is given by the 
sum of two downward cones. 
From the physical point of view, it is important to understand the 
relation of the asymptotic speed and the distance between two cones.

Let us consider a representative example: 
\begin{equation}
 \label{force 5-6}
 f (x) = (R_0 - |x - (r,0)|)_+ + (R_0 - |x + (r,0)|)_+
\end{equation}
with a parameter $r > 0$ and a fixed constant $R_0 > 0$.
If $R_0 \in (1/2,1)$, then we have
\[
 c(r) = 0 \quad
 \mbox{if} \ r \in [0,1-R_0) \cup (R_0,\infty).
\]
On the other hand, if $R_0 > 1$,
then
\[
 c(r) =
 \left\{
 \begin{array}{ll}
  2 (R_0 - 1) & \mbox{if} \ r = 0, \\
  R_0 - 1 & \mbox{if} \ r > R_0.
 \end{array}
 \right.
\]
We examine the above cases with $R_0 = 0.8$ and $R_0 = 1.2$
to verify the mathematical results and
get some predictions for the cases
when $r$ falls within the unclear regime
in the above discussions.
Figure \ref{fig:force5-6} presents the numerical results 
of $r \mapsto c_\triangle (40;r)$
with $R_0 = 0.8$ and $1.2$.
In the right figure (case $R_0 = 1.2$) of Figure \ref{fig:force5-6},
we draw horizontal dashed lines at $c_\triangle = 0.2 = R_0 - 1$
and $c_\triangle = 0.4 = 2(R_0 - 1)$,
and vertical dashed line at $r = 1.2 = R_0$.
Our numerical results follow the mathematical ones.
In particular, one can find
$c_\triangle (40;r) \approx 0$ for $r < 1 - R_0$
if $R_0 = 0.8 \in (1/2,1)$,
and $c_\triangle (T;r) \approx R_0 - 1$ for $r > R_0$ if
$R_0 = 1.2$.
On the other hand, the profiles of $c_\triangle$
within the situation in the unclear regime are asymmetric.
The speed $c_\triangle$ grows slowly when the two sources
of  birth depart from the overlapping situations,
but reduces rapidly when the two sources reaches to
the far apart situations. 

\medskip
\textbf{Question 2.} 
From the observation of Figure \ref{fig:force5-6}, the following points are of interests.
\begin{itemize}
\item[(i)] 
Can we estimate the maximum value of $c(r)$? 
\item[(ii)] 
Does the function $r\mapsto c(r)$ have the maximum 
at only one point? 
\end{itemize}

\begin{figure}[htbp]
 \begin{center}
  \includegraphics[scale=0.8]{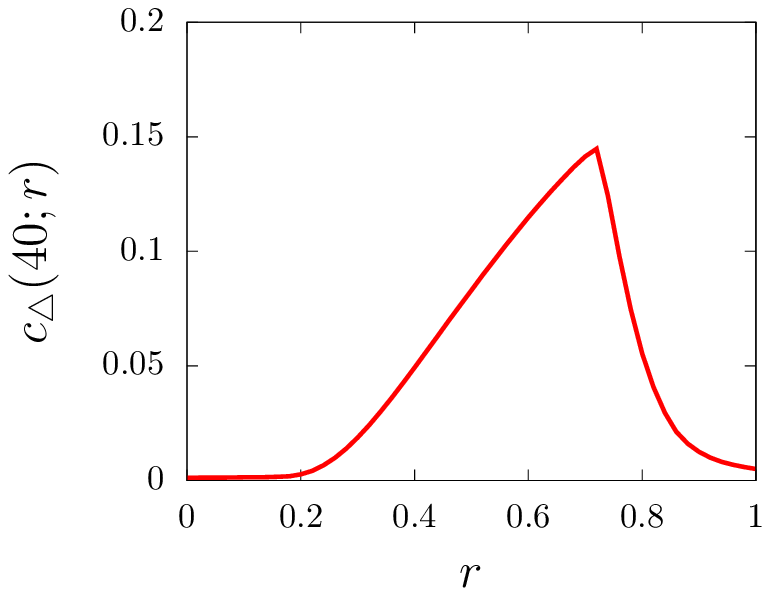}
  \includegraphics[scale=0.8]{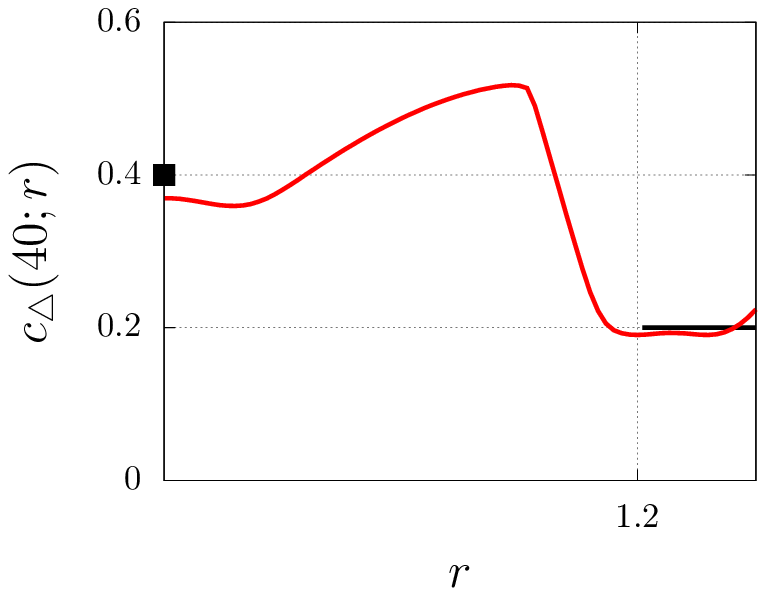}
  \caption{Graph of $c(40;r)$ for \eqref{force 5-6}
  with $R_0 = 0.8$ (left) and $R_0 = 1.2$ (right).
  In the left figure,
  the black dot and lines indicate the mathematical
  results, that is, $c = 2 (R_0 - 1)$ if $r = 0$
  and $c = R_0 - 1$ if $r > R_0$.
  }
  \label{fig:force5-6}
 \end{center}
\end{figure}

%
%

\section{Conclusion}
In this paper, we have established the existence result, Theorem \ref{thm:main}, 
for asymptotic speed of solutions to nonlinear parabolic partial differential equations in a rather general setting. Typical equations which we have in our mind are birth and spread type partial differential equations which are derived by a continuum limit of a Trotter-Kato formula in Section \ref{sec:birth-spread}. 
Three concrete examples, first-order Hamilton-Jacobi equations, forced mean curvature flow, truncated inverse mean curvature flow, are considered as applications of a general framework established in Section \ref{sec:existence}. 

We next investigate qualitative properties of asymptotic speeds. If the front propagation in the horizontal direction is always monotone, that is, $V=g(n(x), \kap(x))>\del$ in \eqref{eq:surface} everywhere for some fixed $\del>0$, then the asymptotic speed is simply the maximum of the source term $f$.

On the other hand, if the front propagation in the horizontal direction is not monotone (e.g., \eqref{FMCF}), 
then a double nonlinear effect coming from the interaction of the nucleation and the surface evolution gives strong influence to the asymptotic speed. 
In the case that $f$ is radially symmetric, we obtain precise formula for the asymptotic speed as \eqref{FMCF} can be reduced to a Hamilton-Jacobi equation with a noncoercive and concave Hamiltonian, which is studied by the optimal control formula. 
In the non-radially symmetric setting, the behavior of solutions is much more involved by the  double nonlinear effect. 
We give several nontrivial properties of the asymptotic speed in this setting 
in Lemmas \ref{lem:c-less-M}, \ref{lem:c-M}, and \ref{lem:obstacle}. 

In Section \ref{sec:numerical}, we give numerical schemes for the forced mean curvature flow
\eqref{FMCF} and  a truncated inverse mean curvature flow \eqref{eq:truncate}, 
and numerical results on asymptotic speed for \eqref{FMCF}.
With the aid of numerical simulation, we raise several concrete questions to be studied in the future.

Finally, in Appendix below, we discuss a volcano formation model, and provide its numerics and also a bit background on inverse mean curvature flow.

\section{Appendix}
In Appendix, we consider \eqref{eq:truncate}, and discuss in the following a model of volcano formation, and some background on inverse mean curvature flow.

\subsection{An explicit solution to \eqref{eq:truncate}}\label{subsec:explicit-sol}
In this subsection, we consider \eqref{eq:truncate}
in the two dimensional setting ($n=2$) with the source of the form
\begin{equation}
 \label{force:single-crater}
 f(x) = \textbf{1}_{\ol{B} (0,R_0)}(x),
\end{equation}
where $R_0 > 0$ is a given constant.
Fix $u_0 \equiv 0$. 
Notice that since $f$ is not continuous, \eqref{f-con} does not hold. 
We construct here the maximal viscosity solution of \eqref{eq:truncate} by using the method in \cite{GMT}. 

Since both $f$ and $u_0$ are radially symmetric, it is reasonable to consider \eqref{eq:truncate} in the radially symmetric setting.
Assume $u(x,t)=\phi(|x|,t)$, and $f(x)=\tilde{f}(|x|)$ for all $x\in \R^2$, and $t \geq 0$, then $\phi$ satisfies
\[
\phi_t-\frac{|\phi_r|}{\chi\left(-\frac{\phi_r}{r|\phi_r|}\right)}=\tilde{f}(r)
\quad\text{ for }  (r,t) \in (0,\infty) \times (0,\infty). 
\]
Under an additional condition that $\phi_r(r,t)\le0$ for $r>0$, the above is simplified into 
\begin{equation}\label{eq:additional}
\phi_t+\frac{\phi_r}{\chi(1/r)}=\tilde{f}(r)
\quad\text{ for }  (r,t) \in (0,\infty) \times (0,\infty).  
\end{equation}
Set 
\[
H(p,r):=\frac{p}{\chi(1/r)}-\tilde{f}(r) \quad \text{ for }  (p,r) \in (-\infty,0]\times (0,\infty).  
\]
Then $p\mapsto H(p,r)$ is linear, hence both
convex and concave for all $r>0$. 
Therefore, we have both of the inf and sup stabilities of viscosity solutions to \eqref{eq:additional} (see \cite[Corollary 7.27]{LMT} for example). 

We fix $\Lam\ge \frac{1}{R_0}$ and $\lam \in (0,\Lam)$ sufficiently small. 
Set $T:= \log \Lam - \log \lam$, and define the function $\phi:[0,\infty)\times[0,\infty)\to\R$ as following.
For $t<T$,
\begin{equation*}
\phi(r,t):=
\left\{
\begin{array}{ll}
t & \text{for all} \ (r,t)\in [0,R_0]\times[0,T),  \\
\max\{t+\log (R_0/r), 0\} & \text{for all} \ (r,t)\in (R_0,\infty)\times[0,T). 
\end{array}
\right. 
\end{equation*}
Moreover, for $t \geq T$, set 
\begin{equation*}
\phi(r,t):=
\left\{
\begin{array}{ll}
t & \text{for all} \ (r,t)\in [0,R_0]\times[T,\infty),  \\
t+\log (R_0/r) & \text{for all} \ (r,t)\in (R_0,1/\lam]\times[T,\infty), \\
\max\{t-\lam r-T, 0\} & \text{for all} \ (r,t)\in (1/\lam,\infty)\times[T,\infty).  
\end{array}
\right. 
\end{equation*}
Notice here that, for $(x,t) \in \ol{B}(0,1/\lam) \times [0,\infty)$, we have
\begin{equation}\label{Fuji-shape}
u(x,t) =\phi(|x|,t)= \min\left\{t, \max \left\{0, t -\log|x| +\log R_0\right\} \right\}.
\end{equation}

It is not hard to check that this is the maximal viscosity solution to \eqref{eq:truncate} on $\R^2 \times [0,\infty)$ by approximating $f$ with a family of continuous functions.  
See \cite{GMT} for details.


%
%

\subsection{A volcano formation model}

The formation of shapes of volcanoes is often explained by a porous medium equation \cite{TS, BDD}.
 Let $h=h(x,t)$ be the height of volcano at a place $x\in\mathbb{R}^2$ and at time $t\in[0,\infty)$.
 A typical evolution of $h$ is modeled as a conservation of mass
\[
	h_t + \operatorname{div} (uh) = 0,
\]
where $u$ is given by Darcy's law
\[
	u = -Dp,
\]
where $p$ is the pressure.
A typical choice of the pressure is $h$ itself.
 Here we set all physical constants just one to clarify the argument.
 The resulting equation for the height is
\[
	h_t = \operatorname{div} (h D h)
\]
or equivalently, 
\[
	2 h_t = \Delta h^2,
\]
which is a particular form of the porous medium equation $h_t=\Delta h^m$.
 This model is proposed for example in \cite[Sections 9.5, 9.6]{TS}.
 In \cite{BDD}, a model with $m=4$ is also proposed.
 To grow a volcano by eruption, we need external supply terms.
 One possible idea is to consider
\[
	h_t = \operatorname{div} (h D h) + \kappa\delta
\]
with $\kappa>0$, where $\delta$ is the Dirac delta function, which is equivalent to give a singular Neumann data at the origin (as in \cite{TS}) when one considers axisymmetric solution.
 In \cite{TS}, a radial self-similar solution of the form $h(x,t)=f\left(x/\sqrt{t}\right)$ is proposed to explain the shape of a volcano.
 Near $x=0$, it is like $h^2 \sim -\log|x|$, a logarithmic shape which looks similar to the shape of a volcano.

However, according to real observation of stratovolcanoes whose shapes are roughly circular cones, this model seems to be insufficient to explain the shape.
 In 1878, a geoscientist J.\ Milne \cite{M1878} measured several volcanoes in Japan including Mt.\ Fuji and observed that the height of each is a logarithmic function from the crater at least in the mountain's breast.
 More precisely, if the crater is located at the origin, then
\[
	h \sim - c_1 \log|x| + c_2
\]
for ``middle range'' of $|x|$, where $c_1$ and $c_2$ are positive constants.
 Much more modern observation is done in \cite{FN} for the shape of Mt.\ Iwate, Japan.
 By these observations, it is rather clear that the model by porous medium equations is not sufficient because $h^2$ is logarithmic, not $h$ itself.

There is an earlier theoretical explanation given by G.\ F.\ Becker \cite{B1885}.
 Let $y$ be the radius at the given height $h$ assuming that a volcano is axisymmetric.
 He proposed that the shape of each volcano satisfies the least variable resistance, which is a minimizer of
\[
	\int^{h_0}_0 \left( y^2 + \alpha(y')^2 \right) \,dh,
\]
where $h_0$ is the height of the volcano and $\alpha$ is a positive constant.
 In \cite{B1885}, the sign in front of $\alpha$ is taken in a wrong way and zero of the region of integration is taken as $h$.
 However, it is not clear why such shape is kept during evolution (eruption).

Our truncated inverse mean curvature flow model does not have physical basis so far, but as discussed in Subsection \ref{subsec:explicit-sol}, when $f=\mathbf{1}_{\ol{B}(0,R_0)}$, \eqref{eq:truncate} has a solution $h = -\log|x| + t+ \log R_0$ in the middle range of $|x|$ (see \eqref{Fuji-shape}), which fits well with what was observed in \cite{M1878, FN}.

 Thus, we propose our equation \eqref{eq:truncate}  as a model of volcano's evolution.
 The inverse mean curvature flow on each level set describes a spreading effect because of the highly viscous lavas while source term represents new eruption, which is supposed to occur regularly like Mt.\ Fuji.

\subsection{Numerical results for the volcano formation model}\label{subsec:IMC}

In this section we give numerical results for \eqref{eq:truncate} with a discretization \eqref{discretized tr-eq}. 
We first consider the function $f$ given by \eqref{force:single-crater}. 
Figure \ref{fig: volcano formation} presents 
the profiles of $u$ and the difference $|u - \phi|$
at $t = 1.25$, and $t=2.50$
with the source size $R_0 = 0.20$ and 
the cut-off parameters $\lambda = 0.5$
and $\Lambda = 0.202 = 1.01R_0$, 
where $u$ is the function computed numerically with \eqref{discretized tr-eq} and 
$\phi$ is the function defined in Section \ref{subsec:explicit-sol}. 
The spatial domain parameters of the calculation
are chosen as $R = 2.56$, and $N=128$, then $\Delta x = 0.02$.
The time span is $\Delta t = 0.025 \times \Delta x^2$.
Our numerical result is very close to the
target $\phi$.
\begin{figure}[htbp]
 \begin{center}
  \includegraphics[scale=1.0]{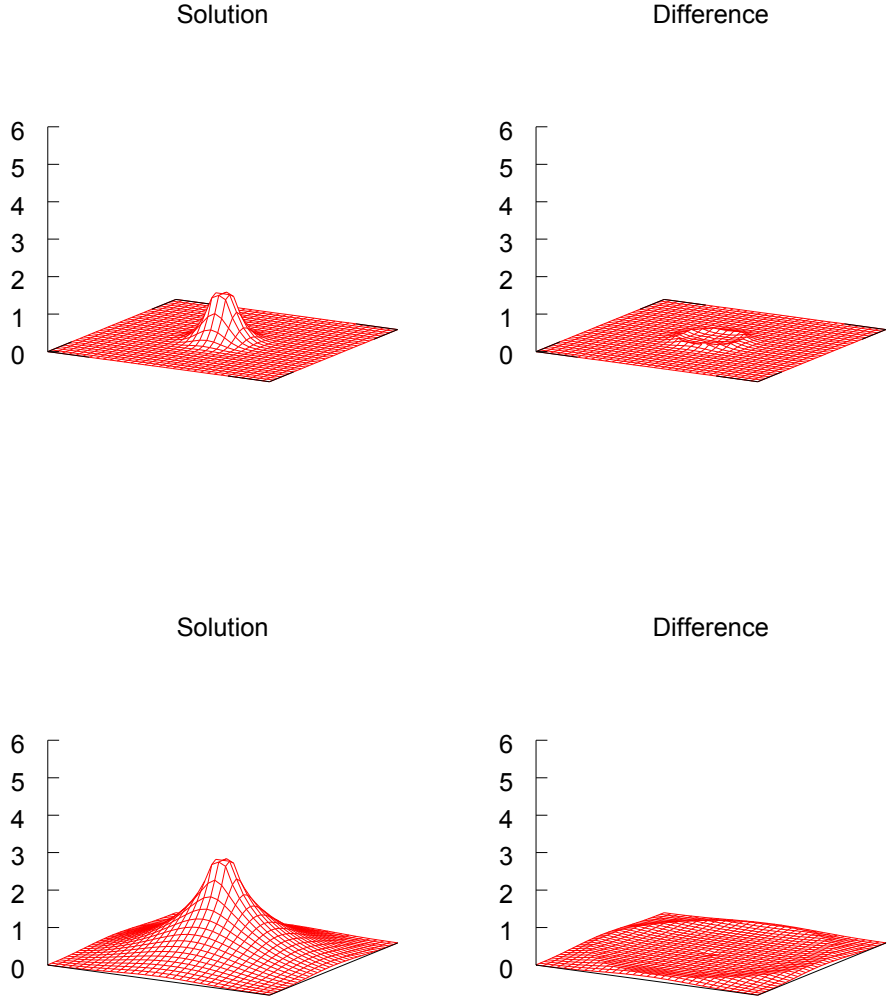}
  \caption{Profile of the solution $u$
  to \eqref{discretized tr-eq} with 
  \eqref{force:single-crater}
  and the difference $|u - \phi|$
  with $\lambda = 0.5$ at $t = 1.25$ (top) 
  and $t=2.5$ (bottom).}
  \label{fig: volcano formation}
 \end{center}
\end{figure}

One may naively think that sending $\lam\to0$, and $\Lam\to\infty$ yields 
\begin{equation} \label{inv-mcf}
  u_t - |D u|
  \left\{
   - \mathrm{div} \left( \frac{D u}{|D u|} \right)
   \right\}^{-1}
  = f
 \quad \mbox{in} \ \mathbb{R}^n \times (0,\infty),
\end{equation}
which is the inverse mean curvature flow equation with a source. 
Equation \eqref{inv-mcf} is also of particular interest.

We discuss here about the difference between \eqref{eq:truncate} and \eqref{inv-mcf}. 
First, if we consider a specific situation in Section \ref{subsec:explicit-sol}, then we realize 
that the truncation from above does not have any effect. 
In general, since the mean curvature of the flow which moves following to \eqref{surface:t-inv} 
may become bigger than $\Lam$, the behavior of solutions to \eqref{eq:truncate} and \eqref{inv-mcf} are slightly different. 

On the other hand, the truncation from below has a subtle and mathematically interesting problem. We first notice that if we do a numerical simulation with a very small $\lambda$, then the difference $|u - \phi|$ is not small. We think that this is because of not only numerical instability but also inconsistency between the solution $u$ to \eqref{discretized tr-eq} and $\phi$ given by \eqref{Fuji-shape} when $\lambda$ is too small, and there might be a hidden reason of this type of inconsistency. When one considers the inverse curvature flow equation \eqref{inv-mcf} in the place with initial data consisting of two non-overlapping circle, a variational solution constructed by \cite{HI} may suddenly jump if two circles are close enough. In other words, the motion may not be local. Since we impose the Neumann boundary condition for numerical simulations, it is the same as periodic case so there are many circles of a positive level set. If these circles are too close it may jump suddenly to the level $c>0$ such that $\{ x \,:\, \ u(x,t) = c \} = \emptyset$. This might be a hidden reason why the numerical solution is not so close to the explicit solution $\phi$. It seems that, when $\lambda = 0$, then there might be a case that there exists no solution to \eqref{eq:truncate} continuous up to initial data even if initial data is continuous or even smooth.

In the simulation in Figure \ref{fig: volcano formation} we choose $\lambda=0.5$. To determine $\lambda = 0.5$ in our simulation, we consider the curvature of level sets $\{ u = c \}$ for every $c \in \mathbb{R}$ for equation \eqref{inv-mcf}. Note that $u$ is radially symmetric if $u_0$ is so. Then, we observe that the curvature of $\{ x  \,:\, u(x, t) = c \}$ for any $c \in \mathbb{R}$ and $t > 0$ is larger than $R \sqrt{2}$ in $[-R,R]^2$ provided that $r \mapsto u (re,t)$ for $e \in S^1$ is monotone decreasing. Hence, it suffices to choose $\lam$ to satisfy $\lam < R \sqrt{2}$ for verifying the volcano formation on $[-R,R]^2$. According to the above discussion and numerical simulations, we choose a fine parameter $\lambda = 0.5$ in this subsection, although we impose the Neumann boundary condition for numerical simulations.

Next, we give a numerical result of volcano formation with two craters. 
Figure \ref{fig: twin5a-profile} presents
a profile of the solution $u$ (left) to \eqref{inv-mcf},
and the difference $|u - \bar{\psi}|$ (right)
at $t=1.25$, $t=2.5$ with source size $R_0 = 0.20$,
cut-off parameters $\lambda = 0.5$, $\Lambda = 0.202 = 1.01 R_0$,
and
\begin{equation}
 \label{force: twin craters}
 f(x) = \textbf{1}_{\ol{B} (-a,R_0) \cup \ol{B}(a, R_0)},
 \quad a = (0.8,0), 
\end{equation}
which is calculated with \eqref{discretized tr-eq}.
The target $\tilde{\psi}$ in this case is chosen as
\begin{align*}
 \tilde{\psi} (x,t) & =
 \min \{ t,
 \max \{ 0, \psi_1 (x,t), \psi_2 (x,t) \} \}, \\
 \psi_1 (x,t) & = t - \log |x - (0.8,0)| + \log R_0, \\
 \psi_2 (x,t) & = t - \log |x + (0.8,0)| + \log R_0.  
\end{align*}

We define the function $\tilde{\psi}$ with an analogy of \eqref{Fuji-shape} and intuition, but we do not know yet whether $\tilde{\psi}$ is the viscosity solution to \eqref{eq:truncate} or not even for a short time. 

Our simulation shows that the solution $u$ is very close to $\bar{\psi}$.
\begin{figure}[htbp]
 \begin{center}
  \includegraphics[scale=1.0]{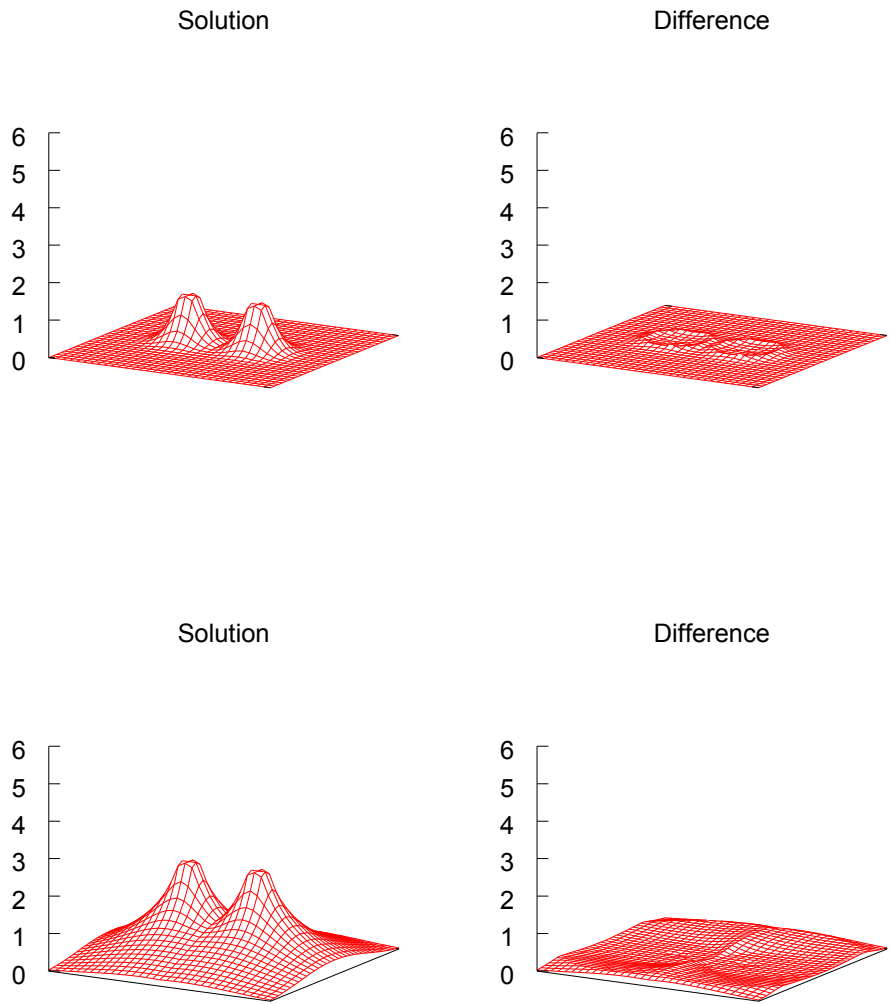}
  \caption{Profile of $u$ (left) 
  to \eqref{discretized tr-eq}
  with \eqref{force: twin craters}
  and $|u - \tilde{\psi}|$ (right)
  with $\lambda = 0.5$
  at $t=1.25$ (top) and $t=2.5$ (bottom).}
  \label{fig: twin5a-profile}
 \end{center}
\end{figure}

\subsection{Inverse mean curvature flow}
In case $f \equiv 0$, \eqref{inv-mcf} becomes
\begin{equation} \label{in-MCF}
	u_t - |D u| \left\{-\operatorname{div} \left(\frac{D u}{|D u|}\right)\right\}^{-1} = 0
	\quad\text{in}\quad	\mathbb{R}^n \times (0,\infty),
\end{equation}
which is the level set flow equation of the inverse mean curvature flow $V=-1/\kap$.
 The inverse mean curvature flow equation is an important tool to prove the Riemann Penrose inequality in general relativity.
 It asserts that the total mass $m_{ADM}$ (often called ADM mass \cite{ADM}) of an asymptotically flat three-dimensional Riemann manifold ($3$-manifold $M$) of nonnegative scalar curvature is bounded from below in terms of each smooth, compact outermost minimal surface in the $3$-manifold.
 An outermost minimal surface is a minimal surface which is not separated from infinity by any other compact minimal surface.
 Hawking \cite{Ha} introduced the Hawking quasi-local mass of a $2$-surface and observe that it approaches to the ADM mass for large coordinate spheres.
 For study of the Hawking mass, Geroch \cite{Ge} first introduced the inverse mean curvature flow and the Hawking mass (sometimes called Geroch mass) is monotone nondecreasing under this flow provided that the surface is connected and the scalar curvature of the ambient $3$-manifold $M$ is nonnegative.
 Jang and Wald \cite{JW} observed that if there was a classical solution of the inverse mean curvature flow starting at the inner boundary and converging to large coordinate sphere as the time tends to $\infty$, the monotonicity result would imply the Penrose inequality since the Hawking mass converges to the ADM mass.

Unfortunately, \eqref{in-MCF} may not have a classical solution.
 To realize this idea, Huisken and Ilmanen \cite{HI} introduced a notion of weak solution,
 which is formulated as a stationary type solution.
 In other words, we set $u(x,t)=v(x)-t$ in \eqref{in-MCF} and observe that $v$ solves 
\[
\begin{cases}
-\operatorname{div} \left( \frac{Dv}{|D v|} \right) = -|D v|\quad &\text{ in } \Omega \subset M,\\
v = 0	\quad &\text{ on } \partial\Omega,
\end{cases}
\]
where the initial surface is $\partial\Omega$, the boundary of a domain $\Omega$ containing the space infinity.
 They introduced a kind of variational weak solution and construct a globally-in-time weak solution having the monotonicity property of the Hawking mass when the initial surface is connected.
 This yields the Penrose inequality $m_{ADM} \geq \sqrt{|N|/16\pi}$, where $|N|$ is an area of connected component of $\partial M$ which is outermost minimal.
 Note that this philosophy is closely related to that in Remark \ref{rem:large-time} about large time behavior of $u$.
 Another proof for the Penrose inequality without assuming that $N$ is connected was given by Bray \cite{Br} by using a different method.

\bibliographystyle{amsplain}
\providecommand{\bysame}{\leavevmode\hbox to3em{\hrulefill}\thinspace}
\providecommand{\MR}{\relax\ifhmode\unskip\space\fi MR }
\providecommand{\MRhref}[2]{%
  \href{http://www.ams.org/mathscinet-getitem?mr=#1}{#2}
}
\providecommand{\href}[2]{#2}

\end{document}